\documentclass[11pt]{amsart}
\usepackage[margin=1.36in]{geometry}
\usepackage{amsmath,amsfonts,amssymb,enumerate,graphicx}
\usepackage{tikz,tikz-cd,mathtools}

\tikzset{
  pics/torus/.style n args={3}{
    code = {
      \providecolor{pgffillcolor}{rgb}{1,1,1}
      \begin{scope}[
          yscale=cos(#3),
          outer torus/.style = {draw,line width/.expanded={\the\dimexpr2\pgflinewidth+#2*2},line join=round},
          inner torus/.style = {draw=pgffillcolor,line width={#2*2}}
        ]
        \draw[outer torus] circle(#1);\draw[inner torus] circle(#1);
        \draw[outer torus] (180:#1) arc (180:360:#1);\draw[inner torus,line cap=round] (180:#1) arc (180:360:#1);
      \end{scope}
    }
  }
}

\let\div\relax
\DeclareMathOperator\div{div}
\DeclareMathOperator\curl{curl}

\DeclareMathOperator\supp{supp}

\DeclareMathOperator\dist{dist}
\newcommand\dd[0]{\partial}
\newcommand\grad[0]{\nabla}
\newcommand\ulin[0]{u^\flat}
\newcommand\unlin[0]{u^\sharp}
\newcommand\unlinI[0]{u^{\sharp,1}}
\newcommand\unlinII[0]{u^{\sharp,2}}
\newcommand\wlin[0]{\omega^\flat}
\newcommand\wnlin[0]{\omega^\sharp}
\newcommand\eq[1]{\begin{align}{#1}\end{align}}
\newcommand\eqn[1]{\begin{align*}{#1}\end{align*}}
\newcommand\Rt[0]{{\mathbb R^3}}
\newcommand\X[3]{X_{{#2};\hspace{.01in}{#3}}^{#1}}
\newcommand\thing[0]{\gamma}
\newcommand\shell[2]{\mathcal S(#1,#2)}
\newcommand\shelltrunc[3]{\mathcal S(#1,#2;#3)}

\newtheorem{theorem}{Theorem}
\newtheorem{prop}{Proposition}
\newtheorem{lemma}{Lemma}
\newtheorem{remark}{Remark}
\newtheorem{example}{Example}

\title[Quantitative regularity for the Navier-Stokes]{Improved quantitative regularity for the Navier-Stokes equations in a scale of critical spaces}

\author{Stan Palasek}
\address{UCLA Department of Mathematics, 520 Portola Plaza, Los Angeles, CA 90025}
\email{palasek@math.ucla.edu}
\thanks{The author is grateful to Terence Tao for many helpful discussions, as well as for providing feedback on a version of the manuscript. We also thank the anonymous referee for their careful reading and comments. This work was partially funded by NSF grant DMS-1764034.}

\begin{document}

\maketitle
 
\begin{abstract}
We prove a quantitative regularity theorem and blowup criterion for classical solutions of the three-dimensional Navier-Stokes equations satisfying certain critical conditions. The solutions we consider have $\|r^{1-\frac3q}u\|_{L_t^\infty L_x^q}<\infty$ where $r=\sqrt{x_1^2+x_2^2}$ and either $q\in(3,\infty)$, or $u$ is axisymmetric and $q\in(2,3]$. Using the strategy of Tao \cite{tao}, we obtain improved subcritical estimates for such solutions depending only on the double exponential of the critical norm. One consequence is a double logarithmic lower bound on the blowup rate. We make use of some tools such as a decomposition of the solution that allows us to use energy methods in these spaces, as well as a Carleman inequality for the heat equation suited for proving quantitative backward uniqueness in cylindrical regions. 
\end{abstract}

\section{Introduction}\label{intro}

In this paper we consider classical solutions of the Navier-Stokes equations
\eqn{
\dd_tu+u\cdot\grad u+\grad p-\Delta u&=0\\
\div u&=0
}
where $u:[0,T]\times\Rt\to\Rt$ and $p:[0,T]\times\Rt\to\mathbb R$ are smooth, and the viscosity has been normalized to $1$ by rescaling. We reformulate the problem in the standard way by applying the Leray projection $\mathbb P=1-\grad\Delta^{-1}\div$, thus eliminating the pressure to obtain
\eq{\label{NS}
\dd_tu+\mathbb P\div u\otimes u-\Delta u=0.
}
An important question in the study of \eqref{NS} is when blowup can be precluded, that is, what conditions one can impose on a solution defined on $[0,T)$ that guarantee it can be continued to a larger interval $[0,T')$ for some $T'>T$. Many such criteria are known; see for example \cite{leray}, \cite{prodi}, \cite{serrin}, \cite{lady}, \cite{bkm}, \cite{ess}, \cite{seregin}, \cite{gkp}, and \cite{tao} among others. With the exception of a few recent results that are in some sense slightly supercritical (for example \cite{pan} and \cite{bpnew}), typical results are critical with respect to the Navier-Stokes scaling, meaning that they assume finiteness of a norm which is invariant when $u(x,t)$ and $p(x,t)$ are replaced by $\lambda u(\lambda x,\lambda^2t)$ and $\lambda^2p(\lambda x,\lambda^2t)$.

A notable example is the criterion of Prodi-Serrin-Ladyzhenskaya \cite{prodi,serrin,lady}, which implies that a necessary condition for blowup is that $\|u\|_{L_t^pL_x^q}$ becomes unbounded, where $\frac2p+\frac3q=1$ and $q\in(3,\infty]$. Extension of this result to the endpoint space $L_t^\infty L_x^3$ remained open until the celebrated work of Escauriaza, Seregin, and {\v{S}}ver{\'a}k \cite{ess} who proved that $\|u(t)\|_{L_x^3}$ becomes unbounded near the blowup time using additional tools including compactness methods combined with unique continuation and backward uniqueness properties of the heat equation. Despite the success of these qualitative methods, they suffer the weakness that they do not imply good quantitative bounds like the ones available in the $q>3$ cases; for example, one does not obtain a lower bound on the norm as the blowup time is approached. Recently, Tao \cite{tao} proved a quantitative version of the $L_t^\infty L_x^3$ regularity criterion, using the same Carleman estimates from \cite{ess} to show directly that if the solution concentrates strongly enough in a small region, then there must be some $L_x^3$ mass in many other spatial regions. In particular, the solution can be controlled in subcritical spaces with bounds on the order of $\exp\exp\exp(\|u\|_{L_t^\infty L_x^3}^C)$ and a blowup rate (in the limit supremum sense) of at least $(\log\log\log(T_*-t)^{-1})^{1/C}$ with $C>0$ a universal constant. These bounds are very strong compared to what can be proved using a compactness-based approach, but are much worse than, say, the $q>3$ case where the norm blows up like a power of $(T_*-t)^{-1}$.

A natural question is whether one can remove any of the exponentials from the bound on $u$, or the logarithms from the blowup rate.  This has been addressed by Barker and Prange \cite{bp} who add the additional assumption $\|u\|_{L_t^\infty L_x^{3,\infty}}\leq M$ on the Lorentz norm. In this case, some of the exponential dependence can be placed onto $M$ instead of $\|u\|_{L_t^\infty L_x^3}$ to obtain a lower bound on the blowup rate comparable to $e^{-e^M}\log(T_*-t)^{-1}$. Their method uses spatial concentration which has the advantage of producing a lower bound on the critical norm locally around the singular point, and which holds at every time near $T_*$ rather than on a sequence $t_n\uparrow T_*$.

One of the main challenges in improving the bounds without assuming additional control on the solution is finding suitable spacetime regions where the solution is regular enough to apply Carleman inequalities, particularly the Carleman inequality related to backward uniqueness of the heat equation. This has been done by using $\epsilon$-regularity criteria, either based on enstrophy localization (see \cite{tao}, following \cite{othertao}) or Caffarelli-Kohn-Nirenberg partial regularity (see \cite{bp}, following \cite{ckn}). In order to find a region for which there is a known backward uniqueness Carleman inequality and where a suitable norm is sufficiently small, it has been necessary to use the pigeonhole principle which is the source of the one of exponentials in the final result.

We will show here that for a particular scale of critical spaces, one can avoid this application of the pigeonhole principle and thus prove quantitative bounds that depend only on a double exponential of the norm of $u$, and consequently a blowup rate of at least $(\log\log(T_*-t)^{-1})^{c}$, with $c>0$ a constant depending only on the parameters of the space. Another aim of this paper is to demonstrate that the quantitative approach to regularity initiated in \cite{tao} has implications for the study of Navier-Stokes solutions with axial symmetry and, more broadly, solutions without any symmetry that lie in critical spaces defined in relation to an axis. The literature on critical regularity criteria for the three-dimensional axisymmetric Navier-Stokes is too vast to list exhaustively, but one can see for instance \cite{chae}, \cite{ss}, \cite{lei}, \cite{pan}, \cite{chen}, \cite{albritton}, \cite{sz}, and \cite{seregin2020}. Let us highlight an important example of such a result due to Koch, Nadirashvili, Seregin, and {\v{S}}ver{\'a}k \cite{knss}, and independently Chen, Strain, Yau, and Tsai \cite{csyt}, which says that (perhaps with some additional mild assumptions), if $u$ is axisymmetric and $ru\in L_{t,x}^\infty$, then the solution is regular, where $r:=\sqrt{x_1^2+x_2^2}$. Another, which is in some ways similar to the results of this paper, is due to Chen, Fang, and Zhang \cite{chen}: if $u$ is axisymmetric, $r^{1-\frac3q}u^\theta\in L_t^\infty L_x^q$, and $\|r^{1-\frac3q}u^\theta\|_{L_t^\infty L_x^q(r\leq\alpha)}$ is sufficiently small for some $q\in[3,\infty)$, then the solution is regular. Here $u^\theta$ is the swirl component of the velocity, ie., the projection onto $(-x_2,x_1)^t$. The two main techniques in use so far to prove theorems like these are a Liouville theorem-based approach as in, for example, \cite{knss}, and De Giorgi-Nash-Moser methods as in, for example, \cite{csyt}. Here we show that similar results can be obtained using the quantitative arguments from \cite{tao}, even while assuming less or no symmetry for the solutions. We will be concerned with solutions of \eqref{NS} satisfying the critical bound
\eq{\label{critical}
\|r^{1-\frac3q}u\|_{L_t^\infty L_x^q}\leq A
}
where $u$ and $q$ fall into one of two cases:
\eq{\label{cases}
\text{either }q\in(3,\infty)\text{, or }u\text{ is axisymmetric and }q\in(2,3].
}
Without loss of generality, let us take $A\geq2$. Then we have the following theorems, which mirror those in \cite{tao} but offer improvements of the quantitative bounds.
\begin{theorem}\label{regtheorem}
If $u:[0,T]\times\Rt\to\Rt$ is a classical solution of \eqref{NS} satisfying \eqref{critical} and \eqref{cases}, then it satisfies the bounds
\eqn{
|\grad^j_xu(t,x)|\leq\exp\exp(A^{O(1)})t^{-\frac{j+1}2}
}
and
\eqn{
|\grad^j_x\omega(t,x)|\leq\exp\exp(A^{O(1)})t^{-\frac{j+2}2}
}
for $t\in(0,T]$, $x\in\Rt$, and $j=0,1$.
\end{theorem}

\begin{theorem}\label{blowuptheorem}
Let $u:[0,T_*)\times\Rt\to\Rt$ be a classical solution to \eqref{NS} which blows up at time $t=T_*$. If $u$ and $q$ satisfy \eqref{cases}, then
\eqn{
\limsup_{t\uparrow T_*}\frac{\|r^{1-\frac3q}u(t)\|_{L_x^q(\mathbb R^3)}}{(\log\log\frac1{T_*-t})^c}=+\infty
}
for a constant $c>0$ depending only on $q$.
\end{theorem}

Let us emphasize that in the case $q\in(3,\infty)$, we do not assume any symmetry on $u$. We also wish to stress that here $r$ is not the distance to the origin, but the distance to the $x_3$-axis. When $q>3$, it should be possible to extend these arguments even to the case where $r$ is replaced with, say, $|x_1|$. However we choose instead to work in the axial setting in order to make the results comparable to other regularity theorems in the literature.

\begin{remark}\label{remark}
The condition that $u$ is axisymmetric can in fact be weakened to just that its magnitude is comparable to an axisymmetric function. In other words, it suffices that there exist $f:\mathbb R\times\Rt\to[0,\infty)$ and $C>0$ such that $f$ is axisymmetric and $C^{-1}f\leq|u|\leq Cf$. Indeed, we will only invoke the axial symmetry assumption by way of Propositions \ref{bernsteinprop} and \ref{localbernsteinprop}. One readily sees by inspecting the proofs that these hold just as well under the weaker hypothesis on $u$. Thus in the arguments that follow we will avoid using any special structures of the \emph{exactly} axisymmetric three-dimensional Navier-Stokes equations.

\end{remark}

One noteworthy special case of Theorems \ref{regtheorem} and \ref{blowuptheorem} is when $q=3$, which is the famous Prodi-Serrin-Ladyzhenskaya endpoint $L_t^\infty L_x^3$. By assuming additionally that $u$ is axisymmetric, we obtain the same result as \cite{tao} but with one fewer $\exp$ or $\log$ in the estimates. Also notable is that when $q$ gets large, we approach the the well-known criterion from \cite{knss} and \cite{csyt} cited above, but without needing to assume any kind of symmetry on $u$. Unfortunately, it seems unlikely that this result can be extended all the way to $q=\infty$ using these techniques. Not only do many of the estimates in this paper degenerate as $q\to\infty$, but $L^\infty$-based critical spaces seem to be out of reach of these quantitative methods since the argument relies on locating concentrations in many different spacial regions which then contribute additively to the critical norm. (See Proposition \ref{main}.)

On the other hand, it seems likely that the $q=2$ case is achievable, although we expect Proposition \ref{regularityprop} to fail at this endpoint and pigeonholing would again be necessary to apply the Carleman estimates. Thus one may have to settle for triple exponential and logarithmic bounds. We can justify this as follows. Although all the conditions defined by \eqref{critical} and \eqref{cases} are critical with respect to the Navier-Stokes scaling, we claim that when $q=3$ or $q=2$ with $u$ axisymmetric, the criticality is homogeneous in the sense that the norms measure all concentrations of the solution identically everywhere in space; on the other hand, if, say, $r^{1-\frac3q}u\in L_t^\infty L_x^q$ where $q>3$ or $u$ is axisymmetric and $q>2$, the norm becomes subcritical far from the $x_3$-axis and supercritical near it. (The opposite would be true if $q<3$ or $q<2$ respectively.) This explains why we can handle these cases without gaining a third $\exp$ or $\log$; indeed, by working sufficiently far from the axis, we can guarantee that the velocity and its derivatives are suitably small compared to the scale of the spacetime region. One can see this phenomenon concretely by considering a concentration of $P_Nu$ at an $x_0\in\Rt$ which lies a distance $r_0$ from the axis. (Refer to Section \ref{notation} for the definition of Littlewood-Paley projections.) Using the same heuristic for \eqref{NS} from \cite[p.\ 67]{othertao}, $u$ behaves essentially as a solution to the heat equation unless the advection term in \eqref{NS} dominates the viscosity, which happens when $|P_Nu(x_0)|\gg N$. By the uncertainty principle, such a concentration must occupy a length scale of at least $N^{-1}$. In the case that $N\gg r_0^{-1}$, the ball $B(x_0,N^{-1})$ does not intersect the $x_3$-axis and therefore, roughly speaking, it contributes at least $(r_0N)^{1-\frac3q}$ to the critial norm $\|r^{1-\frac3q}u\|_{L_t^\infty L_x^q}$. Thus by assuming \eqref{critical} with $q>3$, we expect to be able to rule out nonlinear effects with amplitude much larger than $r_0^{-1}$. If we assume axial symmetry (or more weakly that $|P_Nu|$ is comparable to an axisymmetric function, cf.\ Remark \ref{remark}) and $N\gg r_0^{-1}$, then this concentration exists not just in $B(x_0,N^{-1})$ but also in the torus obtained by rotating this ball around the $x_3$-axis; thus the contribution to $\|r^{1-\frac3q}u\|_{L_t^\infty L_x^q}$ can be strengthened to $(r_0N)^{1-\frac2q}$, and we only need $q>2$ to reach the same conclusion. These heuristics are formalized in the proofs of Propositions \ref{bernsteinprop} and \ref{regularityprop}.

In order to work in these weighted spaces, we employ a useful decomposition $u=\ulin+\unlin$ such that $\ulin$ and $\unlin$ have the following properties:
\begin{itemize}
\item $\ulin$ is almost as regular as a solution to the heat equation; in particular all of its derivatives have good bounds in spaces with sufficiently high integrability. \item $\unlin$ is well-controlled in spaces with integrability between $1$ and $3$.
\item $\unlin$ obeys a forced Navier-Stokes equation that permits energy and enstrophy-type estimates.
\end{itemize}
(See Proposition \ref{hierarchy} for precise statements.) It should be thought of as a somewhat regularized version of $u$ that is close to $u$ near spatial infinity but does not fully capture its largest concentrations, ie., it is the ``flat'' part. On the other hand, $\unlin$ captures these ``sharp'' concentrations and decays more quickly at infinity. This decomposition will be essential in order to apply energy estimates, for example in Propositions \ref{epochs} and \ref{tbsprop} where we need to control $u$ in unweighted $L_x^2$. Similar decompositions, which take advantage of the ``self-improving'' property of the bilinearity in Duhamel's formula, have appeared previously, particularly in the setting of negative regularity Besov spaces; see \cite{cp}, \cite{gkp2}, and \cite{ab}. There is also work of Maremonti and Shimizu \cite{ms} involving a decomposition into the solution of a linear problem and a remainder to which one can apply $L^2$ theory.\footnote{We thank the anonymous referee for these references.}

Once we have the decomposition, the argument can proceed with a strategy analogous to \cite{tao}. The rest of Section \ref{estimates} is concerned with proving that properties of $L_t^\infty L_x^3$ solutions of \eqref{NS} including epochs of regularity (Proposition \ref{epochs}), the bounded total speed property (Proposition \ref{tbsprop}), and back propagation of bubbles of concentration (Proposition \ref{backpropagationprop}) extend to solutions satisfying the conditions \eqref{critical} and  \eqref{cases}. We also show that far from the $x_3$-axis, the solution is regular enough to use a Carleman inequality to propagate concentration forward in time (Proposition \ref{regularityprop}). In Section \ref{carleman}, we prove a backward uniqueness-type Carleman inequality with geometry suited for use with Proposition \ref{regularityprop}. The main difference is that we need to work in cylindrical regions where $r$ and $z$ are localized, rather than annular regions as in \cite{tao}; thus the Carleman inequality in the complement of a ball that was a key ingredient in \cite{ess} and \cite{tao} is not well-suited for our setting. In Section \ref{theorems}, we use a scheme similar to \cite{tao} to prove Theorems \ref{regtheorem} and \ref{blowuptheorem}.

\section{Preliminaries}

\subsection{Notation}\label{notation}

Most of the estimates to follow depend on the choice of $q$, and many of them will deteriorate as either $q\downarrow2$, $q\downarrow3$, or $q\to\infty$. Rather than note this in every instance, we use asymptotic notation $X\lesssim Y$ or $X=O(Y)$ to mean that there is a constant $C(q)$ depending only on $q$ such that $|X|\leq C(q)Y$. One should think of $q$ to be fixed at the beginning of the proof so this does not pose any problems. As in \cite{tao}, we fix a large constant $C_0$ that may depend on $q$. With $A$ as in \eqref{critical}, we define the hierarchy of large constants $A_j=A^{C_0^j}$.

We will occasionally write $x\leq y-$ or $x+\leq y$ to mean $x<y$. This will make it possible to abbreviate a collection of strict and non-strict inequalities. For example, $x\leq a$, $x\leq b$, $x<c$ can be written as $x\leq\min(a,b,c-)$.

If $I\subset\mathbb R$ is a time interval, we use $|I|$ to denote its length. If $\Omega\subset\mathbb R^3$, $|\Omega|$ will denote its three-dimensional Lebesgue measure. If $x_0\in\mathbb R^3$ and $R>0$, we will write $B(x_0,R)$ to denote the closed ball $\{x\in\Rt:|x-x_0|\leq R\}$. As noted in the introduction, if $x\in\Rt$, then $r$ denotes the axial distance in cylindrical coordinates, that is $r:=\sqrt{x_1^2+x_2^2}$. For a specific point, say $p\in\Rt$, we will write its axial coordinate as $r(p):=\sqrt{p_1^2+p_2^2}$. For $0<r_1<r_2$, we define the cylindrical shell $\shell{r_1}{r_2}:=\{x\in\Rt:r_1\leq r\leq r_2\}$ along with the truncated versions $\shelltrunc{r_1}{r_2}M:=\{x\in\shell{r_1}{r_2}:|x_3|\leq M\}$ and $\shelltrunc{r_1}{r_2}{M_1,M_2}:=\{x\in\shell{r_1}{r_2}:M_1\leq|x_3|\leq M_2\}$.

We say a scalar-valued function is axisymmetric if its derivative in the spatial direction $(-x_2,x_1)^t$ vanishes identically. We say a vector-valued function is axisymmetric if each component is axisymmetric when the function is written in cylindrical coordinates around the $x_3$-axis.

When studying the nonlinearity of \eqref{NS}, we will use the symmetrized tensor product
\eqn{
u\odot v:=\frac12(u\otimes v+v\otimes u)
}
for $u,v\in\Rt$, or in coordinates, $(u\odot v)_{ij}=\frac12(u_iv_j+u_jv_i)$. This allows the convenient binomial expansion $(u+v)\otimes(u+v)=u\otimes u+2u\odot v+v\otimes v$.

For $\Omega\subset\mathbb R^n$ and $I\subset\mathbb R$, we will use the Lebesgue norms
\eqn{
\|f\|_{L_x^q(\Omega)}:=\left(\int_\Omega|f(x)|^qdx\right)^{1/q}
}
and
\eqn{
\|f\|_{L_t^pL_x^q(I\times\Omega)}:=\left(\int_I\|f(t,\cdot)\|_{L_x^q(\Omega)}^pdt\right)^{1/p}
}
with the usual modifications if $p=\infty$ or $q=\infty$.

If $\Omega\subset\Rt$ and $\alpha\in\mathbb R$, we define the weighted space $\X q\alpha T(\Omega)$ of smooth vector fields $u:[-T,0]\times\Omega\to\Rt$ such that
\eqn{
\|u\|_{\X q\alpha T(\Omega)}:=\|r^\alpha u\|_{L_t^\infty L_x^q([-T,0]\times\Omega)}<\infty.
}
For brevity we will set $\X q\alpha T:=\X q\alpha T(\Rt)$. Note that for notational convenience, we will work mostly with negative time. The spaces become critical with respect to the Navier-Stokes scaling when $\alpha=\alpha_q$, where
\eqn{
\alpha_q:=1-\frac3q.
}
We record H\"older's inequality for $\X q\alpha T$ spaces, which is immediate from the standard version for $L^p$ spaces: assuming $1\leq p,q,r\leq\infty$, $\alpha,\beta,\gamma\in\mathbb R$, $\frac1p=\frac1q+\frac1r$, and $\alpha=\beta+\gamma$,
\eqn{
\|r^\alpha uv\|_{\X p\alpha T}\leq\|r^\beta u\|_{\X q\beta T}\|r^\gamma v\|_{\X r\gamma T}.
}

For a Schwartz function $f:\Rt\to\mathbb R^n$, we define the Fourier transform
\eqn{
\hat f(\xi)=\int_{\Rt}e^{-ix\cdot\xi}f(x)dx
}
and the Littlewood-Paley projection by the formula
\eqn{
\widehat{P_{\leq N}f}(\xi):=\varphi(\xi/N)\hat f(\xi)
}
where $\varphi:\mathbb R^3\to\mathbb R$ is a radial bump function supported in $B(0,1)$ such that $\varphi\equiv1$ in $B(0,1/2)$. Then let
\eqn{
P_N&:=P_{\leq N}-P_{\leq N/2},\\
P_{>N}&:=1-P_{\leq N},\\
\tilde P_N&:=P_{\leq2N}-P_{\leq N/4}.
}
These all commute with other Fourier multipliers such as $\mathbb P$, $\Delta$ and $e^{t\Delta}$. Estimation of such operators in the weighted spaces $\X p\alpha T$ is the subject of the next subsection.

When a summation is indexed with a capital letter such as $N$ or $R$, it should be taken to range over the dyadic integers $2^\mathbb Z$. Thus we have the shorthand notation
\eqn{
\sum_Nf(N):=\sum_{N\in2^\mathbb Z}f(N),\quad\sum_{A\leq N\leq B}f(N):=\sum_{\{N\in 2^\mathbb Z:A\leq N\leq B\}}f(N),
}
for example.

\subsection{Bernstein-type inequalities}

The following proposition shows that Bernstein's inequality for Fourier multipliers with compactly supported symbols extends naturally to weighted $L^p$ spaces such as $\X p\alpha T$. When working with $u$ controlled in an $\X q{\alpha_q}T$ space with $q<3$ one runs into the difficulty that $\alpha_q<0$. Proposition \ref{bernsteinprop}, as well as many of the estimates for other operators we will derive from it, only hold when the weight on the left-hand side has a smaller power than the one on the right (see Remark \ref{counterexamples}), so it is not clear how one would control the components of $u$ with frequency much larger than $r^{-1}$. Fortunately, in the presence of axial symmetry we can avoid these issues and prove a weighted Bernstein inequality which allows us to exchange some integrability for negative powers of $r$.

\begin{prop}\label{bernsteinprop}
Let $m$ be a Fourier multiplier supported in $B(0,N)$ with $|\grad^jm|\leq MN^{-j}$ for $j=0,1,\ldots,100$. If $1\leq q\leq p\leq\infty$ and either
\begin{enumerate}
\item $\alpha>-\frac2p$, $\beta<\frac2{q'}$, and $\alpha\leq\beta$;
\item $p=\infty$, $\alpha=0$, and $0\leq\beta<\frac2{q'}$; or
\item $q=1$, $\beta=0$, and $-\frac2p<\alpha\leq0$,
\end{enumerate}
then we have
\eq{\label{bernstein}
\|r^\alpha T_mu\|_{L^p}&\lesssim MN^{\frac3q-\frac3p+\beta-\alpha}\|r^\beta u\|_{L^q}.
}
If $|u|$ is axisymmetric, then the conditions $\alpha\leq\beta$, $\beta\geq0$, and $\alpha\leq0$ can be improved to $\alpha\leq\beta+\frac1q-\frac1p$, $\beta\geq-\frac1q$, or $\alpha\leq1-\frac1p$, respectively.
\end{prop}

\begin{proof}
In this proof we make use of the standard non-weighted Bernstein inequalities proved, for example, in \cite[Lemma 2.1]{tao}.

When establishing the case of the proposition in which $|u|$ is axisymmetric, let us assume for the moment that the symbol $m$ is likewise axisymmetric.

We begin by rescaling $x$ and $m$ to make $N=M=1$. Then it clearly suffices to show that the operator $T=r^\alpha T_mr^{-\beta}$ is bounded from $L^q(\Rt)$ to $L^p(\Rt)$. To do so, we decompose it into spatially localized pieces as
\eqn{
T=\sum_{R,S}T_{R,S},\quad T_{R,S}=r^\alpha\chi_RT_mr^{-\beta}\chi_S
}
where $\chi:\mathbb R\to[0,\infty)$ is a smooth function such that the collection $\chi_R(x)=\chi(r/R)$ over $R\in2^\mathbb Z$ forms a partition of unity of $\mathbb R\setminus\{0\}$. More specifically we may choose $\chi_R$ to be supported in $\shell{\frac R2}{\frac{3R}2}$. Then $T_{R,S}$ can be expressed as an integral operator $T_{R,S}f(y)=\int_{\Rt}f(x)K(x,y)dx$ with the kernel
\eqn{
K_{R,S}(x,y)=r_y^\alpha r_x^{-\beta}\chi_R(y)\chi_S(x)\hat m(x-y)
}
satisfying
\eqn{
|K_{R,S}(x,y)|&\lesssim R^\alpha S^{-\beta}\chi_R(y)\chi_S(x)\langle x-y\rangle^{-50}
}
where we let $\langle x\rangle=(1+|x|^2)^{1/2}$. Then, bounding the operator with H\"older's inequality, we have that for $R,S$ such that $\max(R/S,S/R)\geq100$,
\eqn{
\|T_{R,S}\|_{L^q\to L^\infty}&\lesssim\|K_{R,S}\|_{L_y^\infty L_x^{q'}}\\
&\lesssim R^\alpha S^{-\beta}\langle\max(R,S)\rangle^{-50}\|\chi_S(x)\langle x_3-y_3\rangle^{-50}\|_{L_y^\infty L_x^{q'}}\\
&\lesssim R^\alpha S^{-\beta+\frac2{q'}}\langle\max(R,S)\rangle^{-50}
}
and
\eqn{
\|T_{R,S}\|_{L^1\to L^p}&\lesssim\|K_{R,S}\|_{L_x^\infty L_y^p}\\
&\lesssim R^\alpha S^{-\beta}\langle\max(R,S)\rangle^{-50}\|\chi_R(y)\langle x_3-y_3\rangle^{-50}\|_{L_x^\infty L_y^p}\\
&\lesssim R^{\alpha+\frac2p}S^{-\beta}\langle\max(R,S)\rangle^{-50}.
}
Then by interpolation, if $p\geq q$, it follows that
\eqn{
\|T_{R,S}\|_{L^q\to L^p}&\lesssim R^{\alpha+\frac2p}S^{-\beta+\frac2{q'}}\langle\max(R,S)\rangle^{-50}.
}
By essentially the same calculation, if $1/100\leq R/S\leq100$, then
\eqn{
\|T_{R,S}\|_{L^q\to L^p}&\lesssim R^{\alpha-\beta+\frac2p+\frac2{q'}}.
}
Unfortunately this estimate is adequate only when $R,S\lesssim1$,
so we separately consider the case where $R$ and $S$ are comparable and $R,S\gg1$. Fix a $\rho\sim R^{1/10}$ that evenly divides $R/4$. Let $v=r^{-\beta}\chi_Su$. We need to find a spatial region that we can dilate slightly (as required to use the localized Bernstein inequality, again see \cite[Lemma 2.1]{tao}) without drastically increasing the $L^\infty$ norm of $T_mv$. Suppose first we do not have such a region, that is
\eqn{
\|T_mv\|_{L^\infty(\shell{\frac R4}{\frac{7R}4})}\geq2\|T_mv\|_{L^\infty(\shell{\frac R4+\rho}{\frac{7R}4-\rho})}\geq\cdots\geq 2^{\frac R{4\rho}}\|T_mv\|_{L^\infty(\shell{\frac R2}{\frac{3R}2})}.
}
Then taking the left- and right-most ends of the inequality, the ordinary Bernstein inequality implies
\eqn{
\|T_mv\|_{L^\infty(\shell{\frac R2}{\frac{3R}2})}\leq2^{-R^{1/2}}\|T_mv\|_{L^\infty}\lesssim R^{-100}\|v\|_{L^q}.
}
It follows that
\eqn{
\|T_{R,S}u\|_{L^\infty}\lesssim R^{-50}\|u\|_{L^q}
}
which is an adequate estimate to proceed with the argument. Otherwise, there exists an $R_0\in[\frac R2,\frac{3R}4]$ such that
\eqn{
\|T_mv\|_{L^\infty(\shell{R-R_0-\rho}{R+R_0+\rho})}\leq\frac12\|T_mv\|_{L^\infty(\shell{R-R_0}{R+R_0})}.
}
Let $x_0$ be a point in the region $\shell{R-R_0}{R+R_0}$ such that
\eqn{
|T_mv(x_0)|\geq\frac12\|T_mv\|_{L^\infty(\shell{R-R_0}{R+R_0})}.
}
By composing with $P_{\leq10}$ and applying the local Bernstein inequality from \cite{tao}, we have the gradient estimate
\eqn{
\|\grad T_mv\|_{L^\infty(\shell{R-R_0-\frac\rho2}{R+R_0+\frac\rho2})}&\lesssim\|T_mv\|_{L^\infty(\shell{R-R_0-\rho}{R+R_0+\rho})}+\rho^{-50}\|v\|_{L^q}\\
&\lesssim\|T_mv\|_{L^\infty(\shell{R-R_0}{R+R_0})}+R^{-5}\|v\|_{L^q}.
}
Therefore, by the fundamental theorem of calculus,
\eqn{
|T_mv(x)|\geq\frac14\|T_mv\|_{L^\infty(\shell{R-R_0}{R+R_0})}
}
for all
\eqn{
x\in B\Big(x_0,\frac1{O(1)}\frac{\|T_mv\|_{L^\infty(\shell{R-R_0}{R+R_0})}}{\|T_mv\|_{L^\infty(\shell{R-R_0}{R+R_0})}+R^{-5}\|v\|_{L^q}}\Big).
}
Importantly, since $\rho\gg1$ and the radius of this ball is less than $1$, it is contained in $\shell{R-R_0-\frac\rho2}{R+R_0+\frac\rho2}$ where the gradient estimate holds. Without axial symmetry, this implies
\eqn{
\|T_mv\|_{L^q}\gtrsim\|T_mv\|_{L^\infty(\shell{R-R_0}{R+R_0})}\left(\frac{\|T_mv\|_{L^\infty(\shell{R-R_0}{R+R_0})}}{\|T_mv\|_{L^\infty(\shell{R-R_0}{R+R_0})}+R^{-5}\|v\|_{L^q}}\right)^{3/q}.
}
No matter which term in the denominator is larger, we conclude (using the ordinary Bernstein inequality for $T_m$ if the first is larger)
\eqn{
\|T_mv\|_{L^\infty(\shell{R-R_0}{R+R_0})}&\lesssim\|v\|_{L^q}.
}
Now suppose $|u|$, and consequently $|v|$, is axisymmetric. Let $\tilde T_m$ be the operator with kernel $|K(x,y)|$. Then by the triangle inequality, inside the same ball, we have the concentration
\eqn{
|\tilde T_m(|v|)(x)|\geq\frac14\|T_mv\|_{L^\infty(\shell{R-R_0}{R+R_0})}.
}
Thanks to the assumption that $m$ is axisymmetric, one easily computes that $\tilde T_m(|v|)$ is as well. Thus, the bound still holds inside the torus obtained by rotating the ball around the $x_3$-axis. (See Figure \ref{fig}.) Note that within this torus, $r\gtrsim R$; therefore
\eqn{
\|\tilde T_m(|v|)\|_{L^q}&\gtrsim\|T_mv\|_{L^\infty(\shell{R-R_0}{R+R_0})}R^{1/q}\\
&\quad\times\left(\frac{\|T_mv\|_{L^\infty(\shell{R-R_0}{R+R_0})}}{\|T_mv\|_{L^\infty(\shell{R-R_0}{R+R_0})}+R^{-5}\|v\|_{L^q}}\right)^{2/q}.
}
Once again, no matter which term in the denominator is larger, this implies
\eqn{
\|T_mv\|_{L^\infty(\shell{\frac R2}{\frac{3R}2})}\lesssim R^{-\frac1q}\|v\|_{L^q}.
}
Since $\supp\chi_R\subset\shell{\frac R2}{\frac{3R}2}$, it follows that
\eqn{
\|T_{R,S}u\|_{L^\infty}&\lesssim R^\alpha\|T_mr^{-\beta}\chi_Su\|_{L^\infty(\shell{\frac R2}{\frac{3R}2})}\lesssim R^{\alpha-\beta}\|u\|_{L^q},
}
or
\eqn{
\|T_{R,S}u\|_{L^\infty}&\lesssim R^{\alpha-\beta-\frac1q}\|u\|_{L^q}
}
in the presence of axial symmetry. By interpolating with the trivial inequality
\eqn{
\|T_{R,S}u\|_{L^1}&\lesssim R^{\alpha-\beta}\|u\|_{L^1},
}
we obtain, if $q\leq p$,
\eqn{
\|T_{R,S}\|_{L^q\to L^p}&\lesssim R^{\alpha-\beta}
}
or
\eqn{
\|T_{R,S}\|_{L^q\to L^p}&\lesssim R^{\alpha-\beta+\frac1p-\frac1q}
}
in the presence of axial symmetry.

\usetikzlibrary{shapes.multipart}

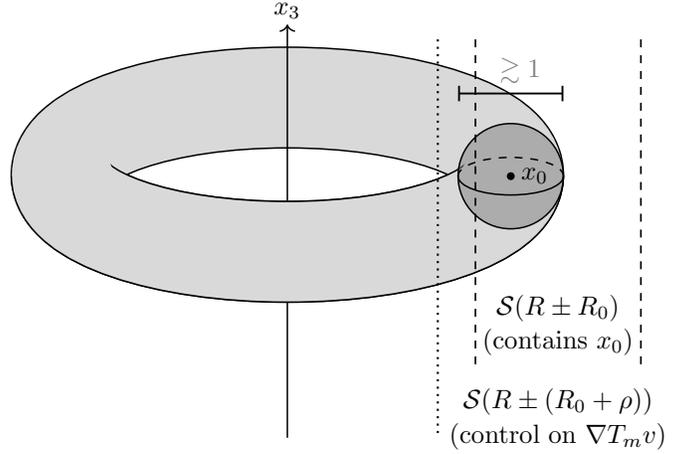
\begin{figure}
\centering
\begin{tikzpicture}[every text node part/.style={align=center}]
    \pic[line width=.2mm,fill=gray!30]{torus={3cm}{.66cm}{70}};
    \draw[->,line width=.2mm](0,-.35)--(0,2);
    \draw[line width=.2mm](0,-3.5)--(0,-1.69);
    \begin{scope}[shift={(.15,-.02)},scale=1,fill opacity=.5,line width=.2mm]
    \draw[line width=.2mm,fill=gray!100](2.82,0)circle(.7cm);
    \draw[line width=.2mm] (2.12,0) arc (180:360:.7 and 0.25);
    \draw[line width=.2mm,dashed] (3.52,0) arc (0:180:.7 and 0.25);
    \filldraw[black,opacity=1] (02.82,0) circle (1.2pt) node[anchor=west] {\small$x_0$};
    \draw[|-|,line width=.25mm,text opacity=1](2.12,1.1)--node[above]{\small$\gtrsim1$}(3.52,1.1);
    \end{scope}
    \node at (0,2.18) {\small$x_3$};
    \draw[dashed,line width=.22mm](2.5,1.8)--(2.5,-2.6);
    \draw[dashed,line width=.22mm](4.7,1.8)--(4.7,-2.6);
    \node at(3.6,-2.0){\small$\mathcal S(R\pm R_0)$\\\small (contains $x_0$)};
    \draw[dotted,line width=.28mm](2,1.8)--(2,-3.5);
    \draw[dotted,line width=.28mm](5.2,1.8)--(5.2,-3.5);
    \node at(3.6,-3.25){\small$\mathcal S(R\pm (R_0+\rho))$\\\small (control on $\grad T_mv$)};
  \end{tikzpicture}
  \caption{We locate a concentration point $x_0$ inside the cylindrical shell $\shell{R-R_0}{R+R_0}$ (delimited by the dashed lines) and, using the gradient estimate on $T_mv$ which holds in the region $\shell{R-R_0-\rho}{R+R_0+\rho}$ (delimited by the dotted lines), one can deduce a comparable estimate holds within a ball which (in the most nontrivial case) has radius at least on the order of $1$. In the case of axial symmetry, we can infer from the pointwise lower bound in the ball that the same lower bound holds within the solid torus obtained by rotating it around the $x_3$-axis. } \label{fig}
\end{figure}

Finally, we can sum over $R,S\in 2^\mathbb Z$ to obtain the desired estimate. Let $\tilde\chi_S$ be a dilated version of $\chi_S$ such that $\chi_S\tilde\chi_S=\chi_S$. Then
\eqn{
Tu=\sum_{\max(R/S,S/R)>100}T_{R,S}(\tilde\chi_Su)+\sum_{1/100\leq R/S\leq100}T_{R,S}(\tilde\chi_Su)
}
where each $x\in\Rt$ lies in the support of boundedly many terms. This implies that without axial symmetry,
\eqn{
\|Tu\|_{L^p}^p&\lesssim\sum_R\left(\sum_{\{S\,:\,\max(R/S,S/R)>100\}}R^{\alpha+\frac2p}S^{-\beta+\frac2{q'}}\langle\max(R,S)\rangle^{-100}\|\tilde\chi_Su\|_{L^q}\right)^p\\
&\quad+\sum_{\substack{1/100\leq R/S\leq100\\\max(R,S)\leq1}}(R^{\alpha-\beta+\frac2p+\frac2{q'}}\|u\|_{L^q})^p+\sum_{\substack{1/100\leq R/S\leq100\\\max(R,S)>1}}(R^{\alpha-\beta}\|\tilde\chi_Su\|_{L^q})^p
}
with the suitable modification if $p=\infty$, in the sense that we are taking an $\ell^p(2^\mathbb Z)$ norm in $R$. When $p<\infty$, the sums converge as geometric series and are bounded by $\|u\|_{L^q}^p$ as long as $\alpha>-\frac2p$, $\beta<\frac2{q'}$, and $\alpha<\beta$. If $p=\infty$, the expression is similarly bounded as long as $0\leq\alpha\leq\beta<\frac2{q'}$. If $\alpha=\beta$, then summability of the last term follows from the embedding $\ell^q(2^\mathbb Z)\to\ell^p(2^\mathbb Z)$ (using $p\geq q$), and the fact that $\sum_{R\geq1}\|\tilde\chi_Ru\|_{L^q}^q\lesssim\|u\|_{L^q}^q$. A similar argument applies to the first term on the right-hand side in the case $q=1$, $\beta=0$.

In the case where $|u|$ has axial symmetry, we carry out an analogous calculation and find the same result except with the last condition relaxed to $\alpha\leq\beta+\frac1q-\frac1p$ thanks to the smaller power of $R$ in the last term.

Now we show how to remove the assumption that $m$ is axisymmetric. Note that $P_{\leq10}$ does have an axisymmetric symbol; moreover $P_{\leq10}T_m=T_m$. Therefore if $|u|$ is axisymmetric,
\eqn{
\|r^\alpha P_{\leq10}u\|_{L^p}\lesssim\|r^\beta u\|_{L^q}
}
assuming $p\geq q$, $\alpha>-\frac2p$, $\beta<\frac2{q'}$, and $\alpha\leq\beta+\frac1q-\frac1p$ (with the appropriate adjustment in the two endpoint cases). Then by the non-axisymmetric version of the theorem,
\eqn{
\|r^\alpha T_mu\|_{L^p}&\lesssim\|r^\alpha P_{\leq10}u\|_{L^p}
}
which yields the desired result. Note that we have $\alpha\leq\beta+\frac1q-\frac1p<2-\frac1q-\frac1p\leq\frac2{p'}$ as required.
\end{proof}

\begin{remark}\label{counterexamples}
Later in the paper, most notably in the proof of Proposition \ref{hierarchy}, we will be applying Proposition \ref{bernsteinprop} in an iterative procedure which will lead to some laborious checking of its hypotheses. The reader may find it illuminating to keep in mind some examples which show why each one is necessary. For simplicity we take $N=1$ and $T_m=P_1$.

Since $P_1u$ is approximately constant on balls of radius $O(1)$, when $p<\infty$, in order for $r^{\alpha p}|P_1u|^p$ to be integrable in such a ball centered on the $x_3$-axis, we need $\alpha p>-2$, or $\alpha>-\frac2p$. Of course when $p=\infty$, there is no such integrability issue as long as $\alpha\geq0$. Next, let $u=\phi(x)/(r^2+\epsilon^2)$ where $u$ is a bump function supported in $B(0,1)$. By the same uncertainty principle heuristic, one finds that $\|r^\alpha P_1u\|_{L^p}$ is comparable to $\log\frac1\epsilon$, but
\begin{equation*}
\|r^\beta u\|_{L^q}\sim\begin{cases}\epsilon^{-\frac2{q'}+\beta},&\beta<\frac2{q'}\\\log^{1/q}\frac1\epsilon,&\beta=\frac2{q'}\\
1,&\beta>\frac2{q'}\end{cases}.
\end{equation*}
By taking $\epsilon$ sufficiently small, we find that the proposition can hold only when either $\beta<\frac2{q'}$ or $q=1$ and $\beta=0$. Let $u$ be a bump function supported in $B(x_0,1)$ where $r(x_0)=R\gg1$. Then \eqref{bernstein} asserts $R^\alpha\lesssim R^\beta$. By taking $R$ sufficiently large, we see $\alpha\leq\beta$. Similarly, consider a smooth axisymmetric function supported in the annulus $\{x\in\Rt:(r-R)^2+x_3^2<1\}$ where $R\gg1$. Then \eqref{bernstein} becomes $R^{\alpha+\frac1p}\lesssim R^{\beta+\frac1q}$ which necessitates $\alpha\leq\beta+\frac1q-\frac1p$.
\end{remark}

As in \cite{tao}, this Bernstein inequality can be localized to a region, at the cost of a global term that can be made small by enlarging the region by a length $\gg N^{-1}$.

\begin{prop}\label{localbernsteinprop}
Let $m$ be a multiplier with $\supp m\subset B(0,N)$ such that
\begin{align*}
|\grad^jm|\leq MN^{-j}
\end{align*}
for $j=0,1,\ldots,2K$ where $K\geq100$. Also let $\Omega\subset\mathbb R^3$ be open and $\Omega_{A/N}=\{x\in\mathbb R^3:\dist(x,\Omega)<A/N\}$. Then
\begin{equation}\begin{aligned}\label{localbernstein}
\|r^{\alpha_1}T_mu\|_{L^{p_1}(\Omega)}&\lesssim_KMN^{\frac3{q_1}-\frac3{p_1}+\beta_1-\alpha_1}\|r^{\beta_1}u\|_{L^{q_1}(\Omega_{A/N})}\\
&\quad+A^{-K}Mr(\Omega)^{\alpha_1-\alpha_2}|\Omega|^{\frac1{p_1}-\frac1{p_2}}N^{\frac3{q_2}-\frac3{p_2}+\beta_2-\alpha_2}\|r^{\beta_2}u\|_{L^{q_2}(\mathbb R^3)}
\end{aligned}\end{equation}
if $p_i\geq q_i$, $p_1\leq p_2$, $\alpha_1\geq\alpha_2$, $\alpha_i>-\frac2{p_i}$, $\beta_i<\frac2{q_i'}$, and $\alpha_i\leq\beta_i$ for $i=1,2$. Here $r(\Omega)$ denotes $\sup\{r:x\in\Omega\}$.

If $|u|$ is axisymmetric, the last condition can be weakened to $\alpha_i\leq\beta_i+\frac1{q_i}-\frac1{p_i}$. As in Proposition \ref{bernstein}, the result extends to the $p_i=\infty$, $\alpha_i=0$ and $q_i=1$, $\beta_i=0$ endpoints.
\end{prop}

We refer to the second term on the right-hand side of \eqref{localbernstein} as the global term. Regardless of what kind of $\X p\alpha T$ control is known for $u$, it is usually possible to make it negligible provided the length scale of $\Omega$ is much smaller than $N^{-1}$.

\begin{proof}
Once again we can rescale to achieve $N=M=1$. Observe by the triangle inequality that it suffices to assume $u$ is supported outside $\Omega_A$, since the part inside can be estimated directly using \eqref{bernstein}. First one uses H\"older's inequality to control the $L^{q_1}$ norm by $L^{q_2}$, as well as the trivial bound $r^{\alpha_1}\leq r^{\alpha_2}r(\Omega)^{\alpha_1-\alpha_2}$. Adopting the notation from the proof of Proposition \ref{bernsteinprop}, we are concerned with estimating convolutions in the form
\eqn{
T_{R,S}(y)=\int_{\mathbb R^3}K_{R,S}(x,y)u(x)dx,
}
but with the additional feature that $y\in\Omega$ and $x\notin\Omega_A$, so $|x-y|\geq A$. Therefore, the estimate for the kernel can be improved to
\eqn{
|K_{R,S}(x,y)|\lesssim_K R^\alpha S^{-\beta}\chi_R(y)\chi_S(x)\langle x-y\rangle^{-50}A^{-K}
}
and one proceeds as in Proposition \ref{bernsteinprop}.
\end{proof}

As a special case of Proposition \ref{bernsteinprop}, with $m=e^{-t|\xi|^2}\psi(\xi/N)$, we get the heat estimate
\begin{align}\label{PNheat}
\|r^\alpha e^{t\Delta}P_N\grad^ju\|_{L^p}\lesssim_je^{-tN^2/20}N^{j+\frac3q-\frac3p+\beta-\alpha}\|r^\beta u\|_{L^q}
\end{align}
under the same assumptions on the parameters. Then summing over $N\in2^\mathbb Z$,
\begin{align}\label{heat}
\|r^\alpha e^{t\Delta}\grad^ju\|_{L^p}\lesssim_jt^{-\frac12(j+\frac3q-\frac3p+\beta-\alpha)}\|r^\beta u\|_{L^q}.
\end{align}

\section{Key estimates for critically bounded solutions}\label{estimates}

\subsection{Decomposition of the solution into $\ulin$ and $\unlin$}

In the spirit of \cite{cp}, \cite{gkp2}, and \cite{ab}, we introduce a decomposition of $u$ related to the Picard iterates of \eqref{NS}. Using Duhamel's principle, we can bootstrap a variety of estimates for $\ulin$ and $\unlin$ in $\X p\alpha T$ spaces. Roughly speaking, $\ulin$ is well-controlled for larger $p$ and $\unlin$ for smaller $p$ (see the discussion in Section \ref{intro}). Later on, this will be essential for using energy methods, in addition to other parts of the argument, for which we need control on $u$ (or at least the roughest component of $u$) in unweighted $L_x^p$ spaces for $p<q$. It will be important too that $\unlin$ satisfies a forced Navier-Stokes equation with some additional linear terms with very regular coefficients.

In practice, this decomposition will be indexed by an integer $n$ which will be needed to become large in order to handle $q\to\infty$ and $q\downarrow2$. Inevitably, many of the estimates to follow on $\ulin_n$ and $\unlin_n$ will deteriorate as $n\to\infty$ but since $q$ is fixed, this is acceptable.

We begin by defining for $t\in[-T,0]$
\eqn{
\ulin_0=0,\quad\unlin_0=u.
}
Let $T_n=(\frac12+2^{-n})T$. Then for $n\geq1$ and $t\in[-T_n,0]$, we iteratively define
\eqn{
\ulin_n(t)&=e^{(t+T_n)\Delta}u(-T_n)-\int_{-T_n}^te^{(t-t')\Delta}\mathbb P\div\ulin_{n-1}\otimes\ulin_{n-1}(t')dt'\\
\unlin_n(t)&=-\int_{-T_n}^te^{(t-t')\Delta}\mathbb P\div(u\otimes u-\ulin_{n-1}\otimes\ulin_{n-1})(t')dt'.
}

By Duhamel's principle applied to the Navier-Stokes on $[-T_n,0]$ we see that for every $n$, these functions sum to $u$. The $n=1$ case corresponds to the vector fields $u^{\text{lin}}$ and $u^{\text{nlin}}$ in \cite{tao}, which would not suffice for our purposes because \eqref{unlinq>3} would only hold for $p\geq\frac q2$ (with a suitable weight).

\begin{prop}\label{hierarchy}
Assume $u$ is a classical solution of \eqref{NS} on $[-T,0]\times\Rt$ satisfying \eqref{critical} and \eqref{cases}. Then we have the following.
\begin{enumerate}
\item For $n=1,2,3,\ldots$ and $t\in[-\frac T2,0]$, $u$ admits the decomposition
\eq{
u=\ulin_n+\unlin_n.
}
\item If $p\geq q$ and either
\eqn{
q>3\quad\text{and}\quad-\frac2p<\alpha\leq\alpha_q
}
or
\eqn{
u\text{ is axisymmetric,}\quad 2<q\leq3,\quad\text{and}\quad-\frac2p<\alpha\leq\alpha_q+\frac1q-\frac1p,
}
then $\ulin_n$ satisfies the bound
\eq{\label{ulin-bound}
\|\grad^j\ulin_n\|_{\X p\alpha{T/2}}\lesssim_nT^{(\alpha-\alpha_p-j)/2}A^{O_n(1)}.
}
These bounds continue to hold at the $p=\infty$, $\alpha=0$ endpoint. For $N\geq T^{-1/2}$, there are also the frequency-localized estimates
\eq{\label{ulinPNq}
\|P_N\ulin_n\|_{\X q{\alpha_q}{T/2}}&\lesssim_ne^{-TN^2/O_n(1)}A^{O_n(1)}
}
and
\eq{\label{ulinPNinfty}
\|P_N\ulin_n\|_{\X\infty0{T/2}}&\lesssim_ne^{-TN^2/O_n(1)}NA^{O_n(1)}.
}
\item If $q\in(2,\infty)$ and $p_0\in(1,3)$, for any $n$ sufficiently large depending on $q$ and $p_0$,
\eq{\label{unlinq>3}
\|\unlin_n\|_{\X p0{T/2}}&\lesssim_n T^{-\alpha_p/2}A^{O_n(1)}
}
for all $p\in[p_0,3)$.
\end{enumerate}
\end{prop}

\begin{proof}
To prove \eqref{ulin-bound}, we claim slightly more strongly that under the stated conditions on $p$, $q$, $n$, and $\alpha$,
\eqn{
\|\grad^j\ulin_n\|_{\X p\alpha{T_{n+1}}}&\lesssim_jT^{(\alpha-\alpha_p-j)/2}A^{O_n(1)}
}
where $T_n=(\frac12+2^{-n})T$ as above. For $\ulin_1$, this is immediate from \eqref{critical} and \eqref{heat}. Suppose we have the desired inequality for some $\ulin_{n-1}$, $n-1\geq1$. From the triangle inequality,
\eqn{
\|\grad^j\ulin_n\|_{\X p\alpha{T_{n+1}}}&\lesssim\|\grad^je^{(t+T_n)\Delta}u(-T_n)\|_{\X p\alpha{T_{n+1}}}\\
&\quad+\int_{-T_n}^t\|r^\alpha\grad^je^{(t-t')\Delta}\mathbb P\div\ulin_{n-1}\otimes\ulin_{n-1}(t')\|_{L_x^p(\Rt)}dt'.
}
The first term is estimated in the same way as $\ulin_1$. For the second term, By H\"older's inequality and \eqref{heat},
\eqn{
\|\grad^je^{(t-t')\Delta}\mathbb P\div\ulin_{n-1}\otimes\ulin_{n-1}\|_{\X p\alpha{T_n}}&\lesssim(t-t')^{-\frac12}\|\grad^j(\ulin_{n-1}\otimes\ulin_{n-1})\|_{\X p\alpha{T_n}}\\
&\lesssim(t-t')^{-\frac12}\sum_{i_1+i_2=j}\Big(\|\grad^{i_1}\ulin_{n-1}\|_{\X{2p}{\alpha/2}{T_n}}\|\grad^{i_2}\ulin_{n-1}\|_{\X{2p}{\alpha/2}{T_n}}\Big).
}
The claimed conditions on $p$ and $\alpha$ are closed under the operation of doubling $p$ and halving $\alpha$, so we achieve the desired bound on $\ulin_n$ upon integrating in time.

Now let us address the frequency-localized estimates. We remark that \eqref{ulin-bound} can also be proven by estimating $P_N\ulin_n$ and summing in $N$, but this is less straightforward than the above method in some endpoint cases.

The $n=1$ case of \eqref{ulinPNq} is immediate from \eqref{PNheat}, and indeed it is true for all $N\geq c_1T^{-1/2}$ (with a constant depending on $c_1$). Suppose that for some $n-1\geq1$ we have the following version of \eqref{ulinPNq} with a slightly wider time interval,
\eqn{
\|P_N\ulin_{n-1}\|_{\X q{\alpha_q}{T_n}}&\lesssim e^{-N^2T/O_n(1)}A^{O_n(1)},
}
for all $N\geq c_1T^{-1/2}$. Then for $N\geq 1000c_1$ and $t\in[-T_{n+1},0]$, by \eqref{PNheat} and \eqref{critical},
\eqn{
\|r^{\alpha_q}P_N\ulin_n(t)\|_{L_x^q(\Rt)}&\lesssim e^{-N^2T/O_n(1)}A\\
&+\int_{-T_n}^te^{-(t-t')N^2/20}N^2\|r^{2\alpha_q}\tilde P_N(\ulin_{n-1}\otimes\ulin_{n-1})(t')\|_{L_x^{q/2}}.
}
Integrating in time, taking a paraproduct decomposition of the nonlinearity, and applying H\"older's inequality, the iterative estimate on $P_N\ulin_{n-1}$, and \eqref{ulin-bound}, the second term becomes
\eqn{
&\|\tilde P_N(P_{>N/100}\ulin_{n-1}\otimes\ulin_{n-1}+P_{\leq N/100}\ulin_{n-1}\otimes P_{>N/100}\ulin_{n-1})\|_{\X{q/2}{2\alpha_q}{T_n}}\\
&\quad\lesssim\sum_{N'>N/100}e^{-(N')^2T/O_n(1)}A^{O_n(1)}\\&\quad\lesssim e^{-N^2T/O_n(1)}A^{O_n(1)}.
}
Then \eqref{ulinPNq} follows by induction on $n$. Note that in order to obtain \eqref{ulinPNq} for a particular $n$ and all $N\geq T^{-1/2}$, one needs to take $c_1$ sufficiently small depending on $n$, since the permissible range for $N$ shrinks by a factor of $1000$ at each step (due to the frequency overlap of the Littlewood-Paley projections). Thus the constant in \eqref{ulinPNq} depends on $n$.

From here, \eqref{ulinPNinfty} is immediate. Indeed, by \eqref{bernstein} and \eqref{ulinPNq},
\eqn{
\|P_N\ulin_n\|_{\X\infty0{T/2}}=\|\tilde P_NP_N\ulin_n\|_{\X\infty0{T/2}}&\lesssim N\|P_N\ulin_n\|_{\X q{\alpha_q}{T/2}}\lesssim_ne^{-N^2T/O_n(1)}NA^{O_n(1)}.
}

Now we turn to estimating $\unlin_n$. The desired estimate \eqref{unlinq>3} is an immediate consequence of the following more general assertion: if $\max\left(1,\frac q{n+1}\right)\leq p\leq\infty$, and either
\eqn{
3<q<\infty,\quad \alpha_p<\alpha<2\min\Big(\frac1{p'},\alpha_q\Big),\quad1\leq n\leq q
}
or
\eqn{
u\text{ is axisymmetric,}\quad2<q\leq 3,\quad\alpha_p<\alpha<\min\left(\frac2{p'},(n+1)\Big(1-\frac2q\Big)-\frac1p\right)
}
then
\eq{\label{unlin}
\|\unlin_n\|_{\X p\alpha{T/2}}\lesssim_nT^{(\alpha-\alpha_p)/2}A^{O_n(1)}.
}
It is straightforward to see that by taking $\alpha=0$ and letting $n$ be large depending on $q$, these conditions reduce to \eqref{cases} and the hypotheses of \eqref{unlinq>3}. To prove \eqref{unlin}, let us decompose $\unlin_n=\unlinI_n+\unlinII_n$ where
\eqn{
\unlinI_n(t)&=-\int_{-T_n}^te^{(t-t')\Delta}\mathbb P\div\unlin_{n-1}\otimes\unlin_{n-1}(t')dt',\\ \unlinII_n(t)&=-2\int_{-T_n}^te^{(t-t')\Delta}\mathbb P\div\ulin_{n-1}\odot\unlin_{n-1}(t')dt'
}
and claim that for $3<q<\infty$, on the slightly larger interval $[-T_n,0]$, we have the desired bound for $\unlinI_n$ $n\geq1$, if
\eqn{
p\geq\frac q{2n},\quad\alpha_p<\alpha<2\min\left(\frac1{p'},2\alpha_q\right)
}
and for $\unlinII_n$ if
\eqn{
p\geq\frac q{n+1},\quad-\frac2p<\alpha<2\min\left(\frac1{p'},\alpha_q\right).
}
The desired result \eqref{unlin} will follow by taking the intersection of these two conditions. The base cases where $n=1$ are immediate from \eqref{heat}, H\"older's inequality, and \eqref{critical} (in fact, $\unlinII_1\equiv0$). On the other hand, the induction on $n$ involves fairly complicated relations between the parameters; thus the reader may find it elucidating to refer to Examples \ref{exgreater} and \ref{exless} which provide concrete examples of the iteration in the $q>3$ and $q\leq3$ cases respectively.

To induct, we assume \eqref{unlin} holds for the $\unlinI$ part for some $n-1\geq1$. By \eqref{bernstein} and H\"older's inequality, for $t\in[-T_n,0]$,
\eqn{
\|r^\alpha\unlinI_n(t)\|_{L_x^p}&\lesssim\int_{-T_n}^t(t-t')^{-\frac12(1+\frac6s-\frac3p+2\beta-\alpha)}\|r^{2\beta}\unlin_{n-1}\otimes\unlin_{n-1}(t')\|_{L_x^{s/2}}dt\\
&\lesssim T^{\frac12(\alpha-\alpha_p)}A^{O_n(1)}
}
assuming there exists an $s\in[2,\infty]$ and $\beta\in\mathbb R$ such that
\eqn{ \alpha\leq2\beta,\quad\frac1p\leq\frac2s\leq1,\quad\alpha>-\frac2p,\quad\beta<1-\frac2s }
which are required for Bernstein's inequality,
\eqn{
\frac6s-\frac3p+2\beta-\alpha<1
}
which is required to integrate in time, and
\eqn{ \frac1s\leq\frac nq,\quad\alpha_s<\beta<2\min\Big(\alpha_q,\frac1{s'}\Big)
}
which are needed so that $\unlin_{n-1}\in\X s\beta{T_{n-1}}$. Letting $\beta=1-\frac2s-\epsilon$ for a positive $\epsilon$ which we will take to be as small as needed depending on the other parameters, the conditions on $s$ reduce to
\eqn{
\max\Big(\frac1{2p},\frac3q-\frac12\Big)\leq\frac1s\leq\min\Big(\frac12,\frac{2-\alpha}4-,\frac{\alpha-\alpha_p}2,\frac nq\Big)
}
which one can check is a nonempty interval if in addition to the relations on $p,q,n,\alpha$ in the hypothesis, we assume
\eq{\label{case1}
\alpha\geq\max\Big(1-\frac2p,\frac6q-\frac3p\Big).
}
Next we let $\beta=\alpha_s+\epsilon$ and the conditions on $s$ become
\eqn{
\max\Big(\frac1{2p},\frac2q-\frac13+\Big)\leq\frac1s\leq\min\Big(\frac12,\frac{2-\alpha}6,\frac nq\Big)
}
which is nonempty and intersects with $[0,1]$ under the hypotheses, along with the additional assumption
\eq{\label{case2} \alpha\leq2-\frac3p. }
Then it is easy to see that as long as $q\geq3$, either \eqref{case1} or \eqref{case2} must be true. This completes the estimate of $\unlinI_n$. Next we have by \eqref{heat} and Holder, for $t\in[-T_n,0]$,
\eqn{
\|r^\alpha\unlinII_n(t)\|_{L_x^p}&\lesssim\int_{-T_n}^t(t-t')^{-\frac12(1+\frac3s-\frac3p+\beta-\alpha)}dt'\|\ulin_{n-1}\|_{\X{\tilde q}0{T_n}}\|\unlin_{n-1}\|_{\X{\frac{s\tilde q}{\tilde q-s}}\beta{T_n}}
}
which implies the desired bound if there exist $\tilde q$, $s$, and $\beta$ such that \eqn{ \alpha\leq\beta,\quad\frac1p\leq\frac1s,\quad\alpha>-\frac2p,\quad\beta<2-\frac2s } for Bernstein, \eqn{ \frac3s-\frac3p+\beta-\alpha<1 } to integrate in time, and
\eqn{
\frac1{\tilde q}\leq\frac1q,\quad0\leq\frac1s-\frac1{\tilde q}\leq\min\Big(1,\frac nq\Big),\quad1-\frac3s+\frac3{\tilde q}<\beta<2\min\Big(\alpha_q,1-\frac1s+\frac1{\tilde q}\Big)
}
to make $\ulin_{n-1}\in\X{\tilde q}0{T_{n-1}}$ and $\unlin_{n-1}\in\X{\frac{s\tilde q}{\tilde q-s}}\beta{T_{n-1}}$.
It suffices to use $\frac1{\tilde q}=\max(\frac1s-\frac nq,0)$. Let us first take $\beta=\alpha$. One can compute that the conditions on $s$ reduce to
\eq{
\label{caseII1} \alpha>1-\frac{3n}q
}
and
\eqn{
\max\Big(\frac1p,\frac{1-\alpha}3+\Big)\leq\frac1s\leq\min\Big(1-\frac\alpha2-,\frac13+\frac1p-,\frac{n+1}q\Big)
}
which is a nonempty interval intersecting with $[0,1]$, assuming \eqref{caseII1} and the original hypotheses on $p,q,\alpha,n$. Now let us instead take $\beta=1-\frac{3n}q+\epsilon$. Then the conditions reduce to
\eqn{
\alpha\leq1-\frac{3n}q,\quad\max\Big(\frac1p,\frac nq\Big)\leq\frac1s\leq\min\Big(\frac12+\frac{3n}{2q}-,\frac\alpha3+\frac1p+\frac nq-,\frac{n+1}q\Big)
}
and one can verify that this is a nonempty interval intersecting with $[0,1]$ if \eqref{caseII1} fails. Thus there always exists a suitable $s$.

Next consider the case with $2<q\leq3$ and $u$ axisymmetric. In the same manner as above, we actually prove the bound for the $n$th iterate on the slightly larger time interval $[-T_n,0]$. As a base for the induction we have
\eqn{
\|\unlin_1\|_{\X p\alpha{T_1}}&\lesssim\int_{-T_1}^t(t-t')^{-\frac12(2+\alpha_p-\alpha)}\|r^{2\alpha_q}u\otimes u(t')\|_{L_x^{q/2}}dt'\lesssim T^{\frac12(\alpha-\alpha_p)}A^2
}
assuming $p\geq\frac q2$, $-\frac2p<\alpha\leq2-\frac4q-\frac1p$, and $\frac3p+\alpha>1$, all of which follow from the assumptions.

Next suppose we have the desired estimate for some $\unlin_{n-1}$, $n-1\geq1$. Proceeding as in the $q>3$ case, the result follows for $\unlinI$ if there exists an $s\in[2,\infty]$ and $\beta\in\mathbb R$ such that
\eqn{
\alpha\leq2\beta+\frac2s-\frac1p,\quad\frac1p\leq\frac2s,\quad\alpha>-\frac2p,\quad\beta<1-\frac2s
}
for the axisymmetric Bernstein inequality,
\eqn{
\frac6s-\frac3p+2\beta-\alpha<1
}
to integrate in time, and
\eqn{
\frac1s\leq\frac nq,\quad\beta<\frac2{s'},\quad\alpha_s<\beta<n\Big(1-\frac2q\Big)-\frac1s
}
so that we have the same bound for $\unlin_{n-1}$. First we let $\beta=\alpha_s+\epsilon$ for a sufficiently small (depending on $\alpha$, $q$, $n$, etc.) $\epsilon>0$. With the given conditions on $\alpha,p,n$, these constraints reduce to
\eqn{
\max\left(\frac1p,1-n\Big(1-\frac2q\Big)+,0+\right)\leq\frac2s\leq\min\left(1-\frac12\Big(\alpha+\frac1p\Big),\frac{2n}q,1\right)
}
which one can verify is a nonempty interval if we assume additionally that
\eq{\label{assumptions}
\alpha\leq2-\frac3p.
}
Next we instead take $\beta=\frac12\Big(\alpha+\frac1p\Big)-\frac1s$, and the conditions reduce to
\eqn{
\max\left(\frac1{2p},\frac14\Big(2-\alpha-\frac1p\Big)+\right)\leq\frac1s\leq\min\left(1-\frac12\Big(\frac1p+\alpha\Big)-,\frac14\Big(1+\frac2p\Big)-,\frac nq,\frac12\right)
}
which is a nonempty interval assuming $\alpha>-\frac1p$. Clearly if this fails, then instead we can conclude by \eqref{assumptions}.

Next, we have
\eqn{
\|r^\alpha\unlinII_n(t)\|_{L_x^p}&\lesssim\int_{-T_n}^t(t-t')^{-\frac12(2+\frac3{s_1}+\frac2{s_2}-\frac3p-\frac2q+\beta-\alpha)}\|r^\beta\unlin_{n-1}(t')\|_{L_x^{s_1}}\\
&\quad\quad\quad\quad\times\|r^{1-\frac2q-\frac1{s_2}}\ulin_{n-1}(t')\|_{L_x^{s_2}}dt'
}
which can be estimated if there exist $s_1$, $s_2$, and $\beta$ such that
\eqn{
\frac1p\leq\frac1{s_1}+\frac1{s_2},\quad-\frac2p<\alpha\leq1+\beta+\frac1{s_1}-\frac1p-\frac2q,\quad\beta<1+\frac2q-\frac2{s_1}-\frac1{s_2}
}
for \eqref{bernstein},
\eqn{
\beta-\alpha+\frac3{s_1}+\frac2{s_2}<\frac3p+\frac2q
}
for integrability in time,
\eqn{
\frac1{s_1}\leq\frac nq,\quad\beta<\frac2{s_1'},\quad\alpha_{s_1}<\beta\leq n\Big(1-\frac2q\Big
)-\frac1{s_1}
}
to control $\unlin_{n-1}$, and
\eqn{
\frac1{s_2}\leq\frac1q,\quad-\frac1{s_2}<1-\frac2q
}
to control $\ulin_{n-1}$. First we take $\beta=\alpha_{s_1}+\epsilon$, and the conditions reduce to the existence of $s_1$ and $s_2$ such that
\eqn{
\frac1p-\frac1{s_1}\leq\frac1{s_2}\leq\min\left(\frac12\Big(\alpha-1+\frac3p+\frac2q\Big)-,\frac1q\right)
}
which one computes is nonempty and intersects $[0,1]$ assuming $s_1$ satisfies
\eqn{
\frac12\max\left(\frac1{p}+,\frac1{p'}-\alpha+,\frac2p,2-(n+1)\Big(1-\frac2q\Big)\right)\leq\frac1{s_1}+\frac1q\leq\min\left(1-\frac12\Big(\alpha+\frac1p\Big),\frac {n+1}q\right).
}
Such an $s_1$ exists in $[0,1]$ assuming
\eq{\label{assum}
\alpha\leq\min\Big(2-\frac3q,2-\frac2q-\frac1p\Big)
}
in addition to the given assumptions.

Next we instead take $\beta=\alpha-1+\frac2q-\frac1{s_1}+\frac1p$. The conditions on $\frac1{s_2}$ reduce to
\eqn{
\max\Big(0,\frac1p-\frac1{s_1}\Big)\leq\frac1{s_2}\leq\min\Big(2-\alpha-\frac1p-\frac1{s_1}-,\frac12+\frac1p-\frac1{s_1},\frac1q\Big)
}
which can be satisfied, along with the other constraints, as long as
\eqn{
\max\Big(0,\frac1p-\frac1q,1-\frac1q-\frac\alpha2-\frac1{2p}+\Big)\leq\frac1{s_1}\leq\min\Big(1,\frac12+\frac1p-,2-\alpha-\frac1p-,\frac nq\Big).
}
There exists such an $s_1$ if, along with the given assumptions, we have
\eqn{
\alpha>2\max\Big(0,1-\frac nq\Big)-\frac2q-\frac1p.
}
One computes that if this constraint fails, then instead we have \eqref{assum}.
\end{proof}

For the reader's convenience, we provide two examples of the iteration for estimating $\unlin_n$ in Proposition \ref{hierarchy}. In these special cases, it becomes routine to verify the many conditions at each step such as the hypotheses of Proposition \ref{bernsteinprop}. Moreover, one can more easily see how the iteration successively makes progress from \eqref{critical} toward the desired estimate.

Let us assume the more straightforward bounds \eqref{ulin-bound} and the $n=1$ case of \eqref{unlin} which follows directly from \eqref{critical}. In fact, in this case the upper bound required on $\alpha$ can be weakened slightly to $\alpha\leq2\alpha_q$ since there is no $\unlinII_1$ contribution.

\begin{example}[$q>3$ case]\label{exgreater}
Let $u$ be as in Proposition \ref{hierarchy}. By rescaling, we may assume $T=1$. With $q=8$, let us prove $\|\unlin_3\|_{\X20{1/2}}\lesssim A^{O(1)}$. We will make use of the estimates
\eqn{
\|\unlin_1\|_{\X4{1/4+\epsilon}{3/4}}\lesssim A^2,\quad
\|\ulin_n\|_{\X p0{T_{n+1}}}&\lesssim A^{2^{n-1}}
}
for $p\geq8$. We can fix, say, $\epsilon=1/100$. Putting the first bound into Duhamel's principle using \eqref{heat} and H\"older's inequality, for $t\in[-3/4,0]$,
\eqn{
\|r^{4\epsilon}\unlinI_2(t)\|_{L_x^3}&\lesssim\int_{-3/4}^t(t-t')^{-1+\epsilon}\|r^{\frac12+2\epsilon}\unlin_1\otimes\unlin_1(t')\|_{L_x^2}dt'\lesssim A^4.
}
Similarly for $\unlinII_2$,
\eqn{
\|r^{4\epsilon}\unlinII_2(t)\|_{L_x^3}&\lesssim\int_{-3/4}^t(t-t')^{-\frac58+\frac32\epsilon}\|r^{\frac14+\epsilon}\unlin_1\odot\ulin_1(t')\|_{L_x^3}dt'\lesssim A^3.
}
so in total we have
\eqn{
\|\unlin_2\|_{\X3{4\epsilon}{3/4}}\lesssim A^4.
}
Again applying this with Duhamel's formula and H\"older's inequality, for $t\in[-5/8,0]$,
\eqn{
\|\unlinI_3(t)\|_{L_x^2}&\lesssim\int_{-5/8}^t(t-t')^{-\frac34-4\epsilon}\|r^{8\epsilon}\unlin_2\otimes\unlin_2(t')\|_{L_x^{3/2}}dt'\lesssim A^8.
}
Next we have
\eqn{
\|\unlinI_2(t)\|_{L_x^{8/3}}&\lesssim\int_{-3/4}^t(t-t')^{-\frac{15}{16}-\epsilon}\|r^{\frac12+2\epsilon}\unlin_1\otimes\unlin_1(t')\|_{L_x^2}dt'\lesssim A^4
}
and
\eqn{
\|\unlinII_2(t)\|_{L_x^{8/3}}&\lesssim\int_{-3/4}^t(t-t')^{-\frac58-\frac\epsilon2}\|r^{\frac14+\epsilon}\unlin_1\odot\ulin_1(t')\|_{L_x^{8/3}}dt'\\
&\lesssim\|r^{\frac14+\epsilon}\unlin_1\|_{\X4{1/4+\epsilon}{3/4}}\|\ulin_1\|_{\X80{3/4}}\lesssim A^3
}
so
\eqn{
\|\unlin_2\|_{\X{8/3}0{3/4}}\lesssim A^4.
}
Finally,
\eqn{
\|\unlinII_3(t)\|_{L_x^2}&\lesssim\int_{-5/8}^t(t-t')^{-\frac12}\|\unlin_2\odot\ulin_2(t')\|_{L_x^2}dt'\lesssim\|\unlin_2\|_{\X{8/3}0{5/8}}\|\ulin_2\|_{\X80{5/8}}\lesssim A^6.
}
In conclusion,
\eqn{
\|\unlin_3\|_{\X20{1/2}}&\lesssim A^8.
}
The argument can be schematized as follows:

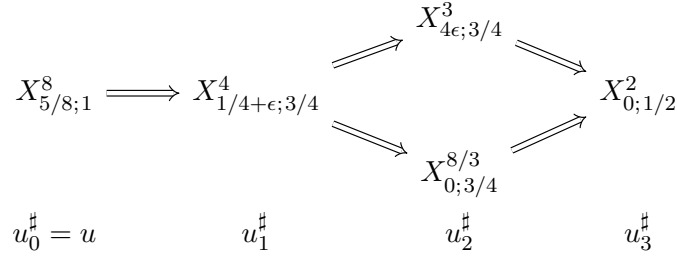
\begin{center}
\begin{tikzcd}[row sep=tiny]
&&\X3{4\epsilon}{3/4}\arrow[dr,Rightarrow]\\
\X8{5/8}1\arrow[Rightarrow,r]&\X4{1/4+\epsilon}{3/4}\arrow[ur,Rightarrow]\arrow[dr,Rightarrow]&&\X20{1/2}\\
&&\X{8/3}0{3/4}\arrow[ru,Rightarrow]&\\[-1.5ex]
\unlin_0=u&\unlin_1&\unlin_2&\unlin_3
\end{tikzcd}
\end{center}

The arrows indicate that we used that $\unlin_{n-1}$ is in one space to prove that $\unlin_n$ is in the next space, and each column corresponds to a particular $n$. The main point of the iteration when $q>3$ is to prove estimates in lower integrability spaces. One can see that the iteration makes progress by using the quadratic nonlinearity to reduce the exponent at each step. The bottleneck in doing so is the $\unlinII$ contribution because of the limited range of $p$ for which \eqref{ulin-bound} holds.
\end{example}

\begin{example}[$2<q\leq3$ axisymmetric case]\label{exless}
Now we let $q=5/2$ and set out to prove $\|\unlin_3\|_{\X{3/2}0{1/2}}\lesssim A^{O(1)}$. We will assume the estimates
\eqn{
\|\unlin_1\|_{\X3\epsilon{3/4}}\lesssim A^2,\quad\|\ulin_1\|_{\X\infty{1/5}{3/4}}\lesssim A,\quad\|\ulin_2\|_{\X{5/2}{-1/5}{5/8}}A^2
}
again for some fixed small $\epsilon>0$. For $t\in[-3/4,0]$,
\eqn{
\|r^{4\epsilon}\unlinI_2(t)\|_{L_x^3}&\lesssim\int_{-3/4}^t(t-t')^{-1+\epsilon}\|r^{2\epsilon}\unlin_1\otimes\unlin_1(t')\|_{L_x^{3/2}}dt'\lesssim A^4
}
and
\eqn{
\|r^{4\epsilon}\unlinII(t)\|_{L_x^3}&\lesssim\int_{-3/4}^t(t-t')^{-\frac1{10}+\frac32\epsilon}\|r^{\frac15+\epsilon}\unlin_1\odot\ulin_1(t')\|_{L_x^3}dt'\\
&\lesssim\|r^\epsilon\unlin_1\|_{\X3\epsilon{3/4}}\|\ulin_1\|_{\X\infty{1/5}{3/4}}\lesssim A^3
}
which imply
\eqn{
\|\unlin_2\|_{\X3{4\epsilon}{3/4}}\lesssim A^4.
}
Similarly,
\eqn{
\|r^{\frac15+4\epsilon}\unlinI_2(t)\|_{L_x^{15/4}}&\lesssim\int_{-3/4}^t(t-t')^{-1+\epsilon}\|r^{2\epsilon}\unlin_1\otimes\unlin_1(t')\|_{L_x^{3/2}}dt'\lesssim A^4
}
and
\eqn{
\|r^{\frac15+4\epsilon}\unlinII_2(t)\|_{L_x^{15/4}}&\lesssim\int_{-3/4}^t(t-t')^{-\frac35+\frac32\epsilon}\|r^{\frac15+\epsilon}\unlin_1\odot\ulin_1(t')\|_{L_x^3}dt'\\
&\lesssim\|\unlin_1\|_{\X3{\epsilon}{3/4}}\|\ulin_1\|_{\X\infty{1/5}{3/4}}\lesssim A^3
}
which imply
\eqn{
\|\unlin_2\|_{\X{15/4}{1/5+4\epsilon}{3/4}}\lesssim A^4.
}
Finally, for $t\in[-5/8,0]$,
\eqn{
\|\unlinI_3(t)\|_{L_x^{3/2}}&\lesssim\int_{-5/8}^t(t-t')^{-\frac12-\epsilon}\|r^{2\epsilon}\unlin_2\otimes\unlin_2(t')\|_{L_x^{3/2}}dt'\lesssim A^8
}
and
\eqn{
\|\unlinII_3(t)\|_{L_x^{3/2}}&\lesssim\int_{-5/8}^t(t-t')^{-\frac12-2\epsilon}\|r^{4\epsilon}\unlin_2\odot\ulin_2\|_{L_x^{3/2}}\lesssim\|\unlin_2\|_{\X{15/4}{1/5+4\epsilon}{5/8}}\|\ulin_2\|_{\X{5/2}{-1/5}{5/8}}\\
&\lesssim A^6
}
which imply
\eqn{
\|\unlin_3\|_{\X{3/2}0{1/2}}\lesssim A^8.
}
Below is the strategy of the iteration.
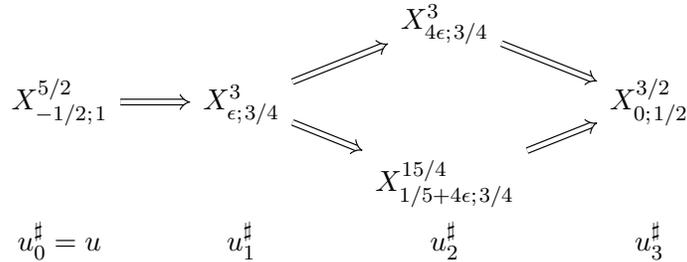
\begin{center}
\begin{tikzcd}[row sep=tiny]
&&\X3{4\epsilon}{3/4}\arrow[dr,Rightarrow]\\
\X{5/2}{-1/2}1\arrow[Rightarrow,r]&\X3{\epsilon}{3/4}\arrow[ur,Rightarrow]\arrow[dr,Rightarrow]&&\X{3/2}0{1/2}\\
&&\X{15/4}{1/5+4\epsilon}{3/4}\arrow[ru,Rightarrow]&\\[-1.5ex]
\unlin_0=u&\unlin_1&\unlin_2&\unlin_3
\end{tikzcd}
\end{center}
Here in the $q<3$ case, the main issue is that we require an estimate with $\alpha=0$, but we are only given \eqref{critical} which has $\alpha_q<0$. The iteration above exploits the fact that \eqref{bernstein} allows one to increase the power $\alpha$ at the cost of increasing the integrability exponent. We can pay this cost thanks to the exponent halving coming from the quadratic nonlinearity and H\"older's inequality.
\end{example}

\subsection{Epochs of regularity}

An essential step in the proof of Proposition \ref{main} consists of using a Carleman estimate (Proposition \ref{uniquecont}) to show that concentrations of the solution near $x=0$ imply additional concentration in regions far from the $x_3$-axis. Since the Carleman estimate demands some pointwise regularity of the solution, it is important that it be applied within an ``epoch of regularity'' which we construct in Proposition \ref{epochs}. This is an extension of Proposition 3.1(iii) in \cite{tao}.

\begin{prop}\label{epochs}
Let $u:[t_0-T,t_0]\times\Rt\to\Rt$ be a classical solution of \eqref{NS} satisfying \eqref{critical} and \eqref{cases}. Then for any interval $I$ in $[t_0-T/2,t_0]$, there is a subinterval $I'\subset I$ with $|I'|\gtrsim A^{-O(1)}|I|$ such that
\eqn{
\|\grad^ju\|_{L_{t,x}^\infty(I\times\mathbb R^3)}\lesssim A^{O(1)}|I|^{-(j+1)/2}
}
and
\eqn{
\|\grad^j\omega\|_{L_{t,x}^\infty(I\times\mathbb R^3)}\lesssim A^{O(1)}|I|^{-(j+2)/2}
}
for $j=0,1$.
\end{prop}

\begin{proof}
By shifting time and rescaling, we may assume that $I=[0,1]$ and $[-1,1]\subset[t_0-T,t_0]$.

For $n$ sufficiently large, \eqref{unlinq>3} implies that
\eqn{
\|\unlin_n\|_{L_t^\infty L_x^p([-1/2,1]\times\mathbb R^3)}\lesssim_pA^{O(1)}
}
for all $p\in[\min(q',\frac q2),3)$. By differentiating the definition of $\unlin_n$ in time, we see that it satisfies
\begin{align*}
\dd_t\unlin_n+\mathbb P\div(u\otimes u-\ulin_{n-1}\otimes\ulin_{n-1})-\Delta\unlin_n&=0.
\end{align*}
Thus, defining
\eqn{
E_0(t):=\frac12\int_\Rt|\unlin_n(x,t)|^2dx,\quad E_1(t):=\frac12\int_\Rt|\grad\unlin_n(x,t)|^2dx,
}
we have the equality
\begin{align*}
\frac d{dt}E_0(t)+\int_\Rt\unlin_n\cdot\div(u\otimes u-\ulin_{n-1}\otimes\ulin_{n-1}-\unlin_n\otimes\unlin_n)dx+2E_1(t)=0,
\end{align*}
using the fact that $\div\unlin_n=0$ and therefore $\int\unlin_n\cdot\div\unlin_n\otimes\unlin_n=0$.
Then integrating in time, using \eqref{unlinq>3} with $p=2$,  expanding the product $u\otimes u-\unlin_n\otimes\unlin_n=2\ulin_n\odot\unlin_n+\ulin_n\otimes\ulin_n$, integrating by parts, and applying Young's inequality,
\begin{align*}
\int_{-\frac12}^1E_1(t)dt&=A^{O(1)}+\int_{-\frac12}^1\int_\Rt\Big(\grad\unlin_n:\ulin_n\odot\unlin_n-\frac12\unlin_n\cdot\div(\ulin_n\otimes\ulin_n-\ulin_{n-1}\otimes\ulin_{n-1}\Big)dxdt\\
&\leq A^{O(1)}+\frac12\int_{-\frac12}^1E_1(t)dt+2\int_{-\frac12}^1\int_{\Rt}\bigg(8|\ulin_n\odot\unlin_n|^2\\
&\quad+\frac12|\unlin_n||\!\div(\ulin_n\otimes\ulin_n-\ulin_{n-1}\otimes\ulin_{n-1})|\Big)dxdt.
\end{align*}
Therefore, by Holder's inequality,
\begin{equation}\begin{aligned}\label{integral}
\int_{-\frac12}^1E_1(t)dt&\lesssim A^{O(1)}+\|\ulin_n\|_{L_{t,x}^\infty([-\frac12,1]\times\Rt)}^2\|\unlin_n\|_{L_t^\infty L_x^2([-\frac12,1]\times\Rt)}^2\\
&\quad+\Big(\|\ulin_{n-1}\|_{L_t^\infty L_x^{2p}([-\frac12,1]\times\Rt)}+\|\ulin_n\|_{L_t^\infty L_x^{2p}([-\frac12,1]\times\Rt)}\Big)^2\|\unlin_n\|_{L_t^\infty L_x^{p'}([-\frac12,1]\times\Rt)}\\
&\lesssim A^{O(1)}.
\end{aligned}\end{equation}
The above inequality is consistent with  the hypotheses of \eqref{ulin-bound} and \eqref{unlinq>3} if we take $p=\max(q,\frac q{q-2})$. Plancherel's theorem then implies
\begin{align}\label{H1}
\sum_NN^2\|P_N\unlin_n\|_{L_t^2L_x^2([-\frac12,1]\times\mathbb R^3)}^2\lesssim A^{O(1)}
\end{align}
and then by Sobolev embedding,
\begin{align}\label{unlin-l6-bound}
\|\unlin_n\|_{L_t^2L_x^6([-\frac12,1]\times\mathbb R^3)}\lesssim A^{O(1)}.
\end{align}

Next, using the equation satisfied by $\unlin_n$, integration by parts and the identity $\unlin_{n-1}=\unlin_n+\ulin_n-\ulin_{n-1}$, we have
\eqn{
\frac{d}{dt}E_1&=-\int_\Rt|\grad^2\unlin_n|^2+\int_\Rt\Delta\unlin_n\cdot\div(\unlin_{n-1}\otimes\unlin_{n-1}+2\ulin_{n-1}\odot\unlin_{n-1})\\
&=-\int_\Rt|\grad^2\unlin_n|^2+\int_\Rt\Delta\unlin_n\cdot\div(\unlin_n\otimes\unlin_n+2\ulin_n\odot\unlin_n)\\
&\quad+\int_\Rt\unlin_n\cdot\Delta\div(\ulin_{n-1}+\ulin_n)\odot(\ulin_n-\ulin_{n-1}).
}
Note that for a vector field $u$, $|\grad^2u|^2$ denotes the quantity $\dd_{ij}u_k\dd_{ij}u_k$. By H\"older's inequality, Sobolev embedding, interpolation, and Gagliardo-Nirenberg,
\eqn{
\|\!\div\unlin_n\otimes\unlin_n\|_{L_x^2(\Rt)}\lesssim\|\unlin_n\|_{L_x^6(\Rt)}\|\grad\unlin_n\|_{L_x^3(\Rt)}&\lesssim\|\grad\unlin_n\|_{L_x^2(\Rt)}^{3/2}\|\grad\unlin_n\|_{L_x^6(\Rt)}^{1/2}\\
&\lesssim E_1(t)^{3/4}\|\grad^2\unlin_n\|_{L_x^2(\Rt)}^{1/2}.
}
By H\"older's inequality, \eqref{ulin-bound}, and \eqref{unlinq>3},
\eqn{
\|\!\div\ulin_n\odot\unlin_n\|_{L_x^2(\Rt)}&\lesssim\|\grad\ulin_n\|_{L_x^\infty(\Rt)}\|\unlin_n\|_{L_x^2(\Rt)}+\|\ulin_n\|_{L_x^\infty(\Rt)}\|\grad\unlin_n\|_{L_x^2(\Rt)}\\
&\lesssim A^{O(1)}(1+E_1(t)^{1/2})
}
and
\eqn{
\|\Delta\div(\ulin_{n-1}+\ulin_n)\odot(\ulin_n-\ulin_{n-1})\|_{L_x^p(\Rt)}&\lesssim A^{O(1)}
}
where again $p=\max(q,\frac q{q-2})$. Combining these estimates for the nonlinearity with Young's inequality, H\"older's inequality, and \eqref{unlinq>3},
\begin{equation}\begin{aligned}\label{dEdt}
\frac{d}{dt}E_1&=-\frac12\int_\Rt|\grad^2\unlin_n|^2+O(E_1(t)^{3/2}\|\grad^2\unlin_n\|_{L_x^2(\Rt)}+A^{O(1)}(1+E_1(t)))\\
&\leq-\frac14\int_\Rt|\grad^2\unlin_n|^2+O(E_1(t)^3+A^{O(1)}(1+E_1(t)).
\end{aligned}\end{equation}
From \eqref{integral}, there exists a time $t_1\in[0,\frac12]$ such that $E_1(t_1)\lesssim A^{O(1)}$. A continuity argument then implies that there is an absolute constant $C$ such that within the interval $I_0=[t_1,t_1+A^{-C}]$, we have $E_1(t)\leq A^C$. More generally we define the truncated intervals $I_j=[t_1+\frac j{10}A^{-C},t_1+A^{-C}]$. Along with \eqref{dEdt}, this implies
\eq{\label{gradsquared}
\int_{I_0}\int_\Rt|\grad^2\unlin_n|^2dxdt\lesssim A^{O(1)}.
}
Using this along with the bound on $E_1$ within the Gagliardo-Nirenberg inequality
\eqn{
\|\unlin_n\|_{L_x^\infty(\Rt)}&\lesssim\|\grad\unlin_n\|_{L_x^2(\Rt)}^{1/2}\|\grad^2\unlin_n\|_{L_x^2(\Rt)}^{1/2},
}
then applying H\"older's inequality in time and \eqref{ulin-bound}, yield
\eqn{
\|u\|_{L_t^4L_x^\infty(I_0\times\Rt)}\leq\|\ulin_n\|_{L_t^4L_x^\infty(I_0\times\Rt)}+\|\unlin_n\|_{L_t^4L_x^\infty(I_0\times\Rt)}\lesssim A^{O(1)}.
}
Duhamel's principle on $I_0$, \eqref{heat}, and Young's inequality give
\eqn{
\|u\|_{L_t^8L_x^\infty(I_1\times\Rt)}&\lesssim\|e^{(t-t_1)\Delta}u(t_1)\|_{L_t^8L_x^\infty(I_1\times\Rt)}+\left\|\int_{t_1}^t(t-t')^{-\frac12}\|u(t')\|_{L_x^\infty(\Rt)}^2dt'\right\|_{L_t^8(I_1)}\\
&\lesssim A^{O(1)},
}
where we truncate the time interval to $I_1$ so the heat propagator in the linear term stays away from the initial time. Bootstrapping and truncating the interval one more time in the same manner, we arrive at
\eq{\label{epoch}
\|u\|_{L_{t,x}^\infty(I_2\times\Rt)}&\lesssim A^{O(1)}.
}
We also have, by \eqref{gradsquared} and Sobolev embedding,
\eq{\label{l2time}
\|\grad\unlin_n\|_{L_t^2L_x^6(I_0\times\Rt)}&\lesssim\|\grad^2\unlin_n\|_{L_t^2L_x^2(I_0\times\Rt)}\lesssim A^{O(1)}
}
which we apply to the Duhamel formula
\eqn{
\grad u(t)=e^{(t-t_1-\frac15A^{-C})\Delta}\grad u(t_1+\frac15A^{-C})-\int_{t_1+\frac15A^{-C}}^te^{-(t-t')\Delta}\grad\mathbb P\div u\otimes u(t')dt'
}
for $t\in I_2$. Using \eqref{heat}, \eqref{epoch}, and \eqref{ulin-bound} and assuming $t\in I_3$,
\eqn{
\|\grad u(t)\|_{L_x^\infty}&\lesssim A^{O(1)}+\int_{t_1+\frac15A^{-C}}^t(t-t')^{-3/4}\|u\cdot\grad\unlin_n(t')\|_{L_x^6}\\
&\quad+(t-t')^{-1/2}\|u\cdot\grad\ulin_n(t')\|_{L_x^\infty}dt'\\
&\lesssim A^{O(1)}\left(1+\int_{t_1}^t(t-t')^{-3/4}\|\grad\unlin_n(t')\|_{L_x^6}dt'\right)
}
and therefore
\eqn{
\|\grad u\|_{L_t^4L_x^\infty(I_3\times\Rt)}&\lesssim A^{O(1)}
}
by fractional integration and \eqref{l2time}. Finally, for $t\in I_4$, by this and \eqref{epoch},
\eqn{
\|\grad u(t)\|_{L_x^\infty(\Rt)}&\lesssim A^{O(1)}+\int_{t_1+\frac3{10}A^{-C}}^t(t-t')^{-1/2}\|u\|_{L_{t,x}^\infty(I_3\times\Rt)}\|\grad u(t')\|_{L_x^\infty(\Rt)}dt'
}
and so
\eq{\label{gradepoch}
\|\grad u\|_{L_{t,x}^\infty(I_4\times\Rt)}&\lesssim A^{O(1)}
}
again by Young's inequality.

The estimates \eqref{epoch} and \eqref{gradepoch} imply regularity of the coefficients of the vorticity equation in $I_4\times\Rt$ and therefore the estimates for $\omega$ and $\grad\omega$ follow by \eqref{vorticity} and parabolic regularity.
\end{proof}

\subsection{Back propagation}

Before we can prove the back propagation we need the following preliminary estimates.
\begin{lemma}\label{pointwiselemma}
We have the pointwise bounds
\eq{\label{pointwise}
|P_Nu|\lesssim AN,\quad|\grad P_Nu|\lesssim AN^2,\quad|\dd_tP_Nu|\lesssim A^2N^3
}
and
\eq{
|P_N\omega|\lesssim AN^2,\quad|\grad P_N\omega|\lesssim AN^3,\quad|\dd_tP_N\omega|\lesssim A^2N^4.
}
\end{lemma}

\begin{proof}
By \eqref{bernstein},
\eqn{
|\grad^jP_Nu|\lesssim N^{j+3/q+\alpha_q}\|r^{\alpha_q}u\|_{L^q}\lesssim AN^{1+j}.
}
Applying $P_N$ to \eqref{NS} and again using \eqref{bernstein} and H\"older's inequality,
\eqn{
|\dd_tP_Nu|&\leq\|P_N\mathbb P\div(u\otimes u)\|_{L_x^\infty}+\|P_N\Delta u\|_{L_x^\infty}\\
&\lesssim N^{1+\frac6q+2\alpha_q}\|r^{2\alpha_q}u\otimes u\|_{L_x^{q/2}}+N^{2+\frac3q+\alpha_q}\|r^{\alpha_q}u\|_{L_x^q}\\
&\lesssim A^2N^3.
}
The vorticity estimates are proved in the same way, with an extra $\xi\times$ in the multiplier. Note that the weights satisfy the Bernstein inequality when $q\geq3$ because then $\alpha_q\geq0$, and when $2<q<3$ because then $\alpha_q>-\frac1q$. 
\end{proof}

The ``bounded total speed property'' (see \cite{othertao}) is essential for iterating the back propagation. Proposition \ref{tbsprop} is an extension of the version that appears in \cite{tao}.

\begin{prop}\label{tbsprop}
Let $u$ be as in Proposition \ref{epochs}. For any time interval $I\subset[t_0-T/2,t_0]$, we have
\eq{\label{tbs}
\|u\|_{L_t^1L_x^\infty(I\times\Rt)}&\lesssim A^{O(1)}|I|^{1/2}.
}
\end{prop}

\begin{proof}
As in the proof of Proposition \ref{epochs}, assume without loss of generality that $I=[0,1]$. Once again we let $n$ be sufficiently large so that
\eqn{
\|\unlin_n\|_{L_t^\infty L_x^p([-1/2,1]\times\mathbb R^3)}&\lesssim_pA^{O(1)}
}
for all $p\in[q',3)$.

From the equation for $\unlin$ we have
\begin{align*}
P_N\unlin_n(t)=-\int_{-1/2}^tP_Ne^{(t-t')\Delta}\mathbb P\div\tilde P_N(u\otimes u-\ulin_{n-1}\otimes\ulin_{n-1})(t')dt'
\end{align*}
and so
\begin{align*}
&\|P_N\unlin_n\|_{L_t^1L_x^\infty([-1/2,1]\times\Rt)}\\
&\quad\lesssim\left\|\int_{-1/2}^tNe^{-N^2(t-t')/20}\|\tilde P_N(u\otimes u-\ulin_{n-1}\otimes\ulin_{n-1})(t')\|_{L_x^\infty(\Rt)}dt'\right\|_{L_t^1([-1/2,1])}\\
&\quad\lesssim N^{-1}\|\tilde P_N(u\otimes u-\ulin_{n-1}\otimes\ulin_{n-1})\|_{L_t^1L_x^\infty([-1/2,1]\times\Rt)}.
\end{align*}
We split $u\otimes u=\unlin_n\otimes\unlin_n+2\ulin_n\odot\unlin_n+\ulin_n\otimes\ulin_n$ and estimate, by \eqref{bernstein} and \eqref{ulin-bound},
\begin{align*}
\|\tilde P_N(\ulin_n\otimes\ulin_n)\|_{L_t^1L_x^\infty([-1/2,1]\times\Rt)}\lesssim\|\ulin_n\|_{L_{t,x}^\infty([-1/2,1]\times\Rt)}^2&\lesssim A^{O(1)}
\end{align*}
and similarly for $\ulin_{n-1}\otimes\ulin_{n-1}$. By \eqref{bernstein}, H\"older's inequality in time, and \eqref{unlin-l6-bound},
\begin{align*}
\|\tilde P_N(\ulin_n\odot\unlin_n)\|_{L_t^1L_x^\infty([-1/2,1]\times\Rt)}&\lesssim N^{1/2}\|\tilde P_N(\ulin_n\odot\unlin_n)\|_{L_t^1L_x^6([-1/2,1]\times\Rt)}\\
&\lesssim N^{1/2}\|\ulin_n\|_{L_{t,x}^\infty([-1/2,1]\times\Rt)}\|\unlin_n\|_{L_t^2L_x^6([-1/2,1]\times\Rt)}\\
&\lesssim A^{O(1)}N^{1/2}.
\end{align*}
Finally, we decompose $\unlin_n=P_{\leq N}\unlin_n+P_{>N}\unlin_n$ and estimate the three terms that appear when $\unlin_n\otimes\unlin_n$ is expanded. By \eqref{bernstein} and H\"older's inequality,
\begin{align*}
\|\tilde P_N(P_{\leq N}\unlin_n\otimes P_{\leq N}\unlin_n)\|_{L_t^1L_x^\infty([-1/2,1]\times\Rt)}&\lesssim\|P_{\leq N}\unlin_n\|_{L_t^2L_x^\infty([-1/2,1]\times\Rt)}^2,\\
\|\tilde P_N(P_{\leq N}\unlin_n\odot P_{>N}\unlin_n)\|_{L_t^1L_x^\infty([-1/2,1]\times\Rt)}&\lesssim N^{3/2}\|P_{\leq N}\unlin_n\odot P_{>N}\unlin_n\|_{L_t^1L_x^2([-1/2,1]\times\Rt)}\\
&\lesssim N^{3/2}\|P_{\leq N}\unlin_n\|_{L_t^2L_x^\infty([-1/2,1]\times\Rt)}\\
&\quad\times\|P_{>N}\unlin_n\|_{L_t^2L_x^2([-1/2,1]\times\Rt)},\\
\|\tilde P_N(P_{>N}\unlin_n\otimes P_{>N}\unlin_n)\|_{L_t^1L_x^\infty([-1/2,1]\times\Rt)}&\lesssim N^3\|P_{>N}\unlin_n\otimes P_{>N}\unlin_n\|_{L_t^1L_x^1([-1/2,1]\times\Rt)}\\
&\lesssim N^3\|P_{>N}\unlin_n\|_{L_t^2L_x^2([-1/2,1]\times\Rt)}^2.
\end{align*}
In total, by Young's inequality,
\begin{align*}
\|\tilde P_N(\unlin_n\otimes\unlin_n)\|_{L_t^1L_x^\infty([-1/2,1]\times\Rt)}&\lesssim\|P_{\leq N}\unlin_n\|_{L_t^2L_x^\infty([-1/2,1]\times\Rt)}^2\\
&\quad+N^3\|P_{>N}\unlin_n\|_{L_t^2L_x^2([-1/2,1]\times\Rt)}^2.
\end{align*}
Inserting this into the estimate for $\unlin_n$,
\begin{equation}\begin{aligned}\label{guy}
\|P_N\unlin_n\|_{L_t^1L_x^\infty([-1/2,1]\times\Rt)}&\lesssim N^{-1}\|P_{\leq N}\unlin_n\|_{L_t^2L_x^\infty([-1/2,1]\times\Rt)}^2\\
&\quad+N^2\|P_{>N}\unlin_n\|_{L_t^2L_x^2([-1/2,1]\times\Rt)}^2+A^{O(1)}(N^{-1}+N^{-1/2}).
\end{aligned}\end{equation}
By \eqref{bernstein} and Cauchy-Schwarz,
\begin{align*}
\|P_{\leq N}\unlin_n\|_{L_t^2L_x^\infty([-1/2,1]\times\Rt)}^2&\lesssim\left(\sum_{N'\leq N}(N')^{3/2}\|P_{N'}\unlin_n\|_{L_t^2L_x^2([-1/2,1]\times\Rt)}\right)^2\\
&\lesssim N^{1/2}\sum_{N'\leq N}(N')^{5/2}\|P_{N'}\unlin_n\|_{L_t^2L_x^2([-1/2,1]\times\Rt)}^2
\end{align*}
and by Plancherel's theorem,
\begin{align*}
\|P_{>N}\unlin_n\|_{L_t^2L_x^2([-1/2,1]\times\Rt)}^2\lesssim\sum_{N'>N}\|P_{N'}\unlin_n\|_{L_t^2L_x^2([-1/2,1]\times\Rt)}^2.
\end{align*}
Plugging these into \eqref{guy}, we obtain an estimate for the high frequency component,
\begin{align*}
\|P_{\geq1}\unlin_n\|_{L_t^1L_x^\infty([-1/2,1]\times\Rt)}&\lesssim \sum_{N\geq1}\Bigg(N^{-1/2}\sum_{N'\leq N}(N')^{5/2}\|P_{N'}\unlin_n\|_{L_t^2L_x^2([-1/2,1]\times\Rt)}^2\\
&\quad+N^2\sum_{N'>N}\|P_{N'}\unlin_n\|_{L_t^2L_x^2([-1/2,1]\times\Rt)}^2\Bigg)+A^{O(1)}\\
&\lesssim \sum_NN^2\|P_N\unlin_n\|_{L_t^2L_x^2([-1/2,1]\times\Rt)}^2+A^{O(1)}\\
&\lesssim A^{O(1)}.
\end{align*}
by \eqref{H1}. For the remaining parts of $u$, by H\"older's inequality in time, \eqref{bernstein}, \eqref{ulin-bound}, and \eqref{unlin},
\begin{align*}
\|\ulin_n\|_{L_t^1L_x^\infty([-1/2,1]\times\Rt)}&\lesssim\|\ulin_n\|_{L_{t,x}^\infty([-1/2,1]\times\Rt)}\lesssim A^{O(1)}
\end{align*}
and
\begin{align*}
\|P_{<1}\unlin_n\|_{L_t^1L_x^\infty([-1/2,1]\times\Rt)}\lesssim\|\unlin_n\|_{L_t^\infty L_x^2([-1/2,1]\times\Rt)}\lesssim A^{O(1)}
\end{align*}
which completes the proof.
\end{proof}

Now we can prove the back propagation proposition from \cite{tao} with the more general critical control on $u$.
\begin{prop}\label{backpropagationprop}
Let $u$ be as in Proposition \ref{epochs}. Suppose there exist $(t_1,x_1)\in[t_0-\frac T2,t_0]\times\Rt$ and $N_1\geq A_3T^{-\frac12}$ such that
\eq{\label{lb}
|P_{N_1}u(t_1,x_1)|\geq A_1^{-1}N_1.
}
Then there exists $(t_2,x_2)\in[t_0-T,t_1]\times\Rt$ and $N_2\in[A_2^{-1}N_1,A_2N_1]$ such that
\eqn{
A_3^{-1}N_1^{-2}\leq t_1-t_2\leq A_3N_1^{-2},
}
\eqn{
|x_2-x_1|\leq A_4N_1^{-1},
}
and
\eqn{
|P_{N_2}u(t_2,x_2)|\geq A_1^{-1}N_2.
}
\end{prop}

\begin{proof}
First consider $q>3$. We scale and translate so that $N_1=1$ and $t_1=0$. Then $[-2A_3,0]\subset[t_0-T,t_0]$. Then by assumption we have
\eq{\label{lowerbound}
|P_1u(0,x_1)|\geq A_1^{-1}.
}
Assume for contradiction that the claim fails, which would imply
\eqn{
\|P_Nu\|_{L_{t,x}^\infty([-A_3,-A_3^{-1}]\times B(x_1,A_4))}\leq A_1^{-1}N
}
for $N\in[A_2^{-1},A_2]$. From the pointwise bound on $\dd_tP_Nu$ and the fundamental theorem of calculus, the time interval can be enlarged up to $t=0$,
\eq{\label{linfty}
\|P_Nu\|_{\X\infty0{A_3}(B(x_1,A_4))}\lesssim A_1^{-1}N+A_3^{-1}A^2N^3\lesssim A_1^{-1}N.
}
For $t\in[-A_3,0]$, Duhamel's formula, H\"older's inequality for the linear term, and \eqref{bernstein} give us
\eqn{
\|r^\alpha P_Nu(t)\|_{L_x^{q/2}(B(x_1,A_4))}&\leq A_4^{3/q}\|r^\alpha e^{(t+2A_3)\Delta}P_Nu(-2A_3)\|_{L_x^q(B(x_1,A_4))}\\
&\quad+\int_{-2A_3}^t\|r^\alpha e^{(t-t')\Delta}P_N\div u\otimes u(t')\|_{L_x^{q/2}(\mathbb R^3)}dt'\\
&\lesssim A_4^{3/q}e^{-N^2A_3/20}N^{\alpha_q-\alpha}A+\int_{-2A_3}^te^{-(t-t')N^2/20}N^{1+2\alpha_q-\alpha}A^2dt'
}
assuming $-\frac2q<\alpha\leq\alpha_q$. Therefore, for $N\geq A_2^{-1}$,
\eq{
\|P_Nu\|_{\X{q/2}\alpha{A_3}(B(x_1,A_4))}\lesssim A^2N^{1-\frac6q-\alpha}.
}
Starting from this base case, we claim inductively that
\eq{\label{one}
\|P_Nu\|_{\X{q/n}\alpha{T_n}(B_n)}&\lesssim N^{1-\frac{3n}q-\alpha}A^{O_n(1)}
}
for all $N\geq A_2^{-\frac12-\frac1n}$, where $T_n=(\frac12+\frac1n)A_3$ and $B_n=B(x_1,(\frac12+\frac1n)A_4)$, if $2\leq n\leq\min(q,\frac{q+5}2-)$ and $-\frac2q<\alpha\leq\min(\alpha_q,2-\frac{2n}q-)$. Suppose \eqref{one} holds for some $n-1\geq2$. For $t\in[-T_n,0]$,
\eqn{
\|r^\alpha P_Nu(t)\|_{L_x^{q/n}(B_n)}&\leq\|r^\alpha e^{(t+T_{n-1})\Delta}P_Nu(-T_{n-1})\|_{L_x^{q/n}(B_n)}\\
&\quad+\int_{-T_{n-1}}^t\|r^\alpha e^{(t-t')\Delta}P_N\div\tilde P_N(u\otimes u)(t')\|_{L_x^{q/n}(B_n)}dt'.
}
The linear term can be handled exactly as in the previous case, using H\"older and \eqref{bernstein}. Then by \eqref{localbernstein} and a paraproduct decomposition of the nonlinearity,
\eqn{
\|r^\alpha e^{(t-t')\Delta}P_N\div(u\otimes u)\|_{L_x^{q/n}(B_n)}&\lesssim_ne^{-(t-t')N^2/20}N^{1-\alpha}\\
&\hspace{-1.6in}\times\bigg(N^{\alpha_q+\beta}\|r^{\alpha_q+\beta}(P_{>N/100}u\otimes u+P_{\leq N/100}u\otimes P_{>N/100}u)\|_{L_x^{q/n}(B_{n-1})}\\
&\quad\hspace{-1.6in}+(NA_4)^{-50n}A_4^{\frac3q(n-2)}N^{2\alpha_q}\|r^{2\alpha_q}u\otimes u\|_{L_x^{q/2}(\mathbb R^3)}\bigg),
}
assuming additionally that $\alpha-\alpha_q\leq\beta<1-\frac{2n-3}q$. This implies
\eqn{
\|P_Nu\|_{\X{q/n}\alpha{T_n}(B_n)}&\lesssim_n(NA_4)^{-10}+N^{-\frac3q+\beta-\alpha}\bigg(\|(P_{>N/100}u)\otimes u\|_{\X{q/n}{\alpha_q+\beta}{T_{n-1}}(B_{n-1})}\\
&\quad+\|(P_{\leq N/100}u)\otimes(P_{>N/100}u)\|_{\X{q/n}{\alpha_q+\beta}{T_{n-1}}(B_{n-1})}\bigg)\\
&\quad+(NA_4)^{-50n}A_4^{\frac3q(n-2)}N^{1-\frac6q-\alpha}A^2.
}
For the first nonlinear term, for $N\geq A_2^{-\frac12-\frac1n}$,
\eqn{
\|P_{>N/100}u\otimes u\|_{\X{q/n}{\alpha_q+\beta}{T_{n-1}}(B_{n-1})}&\lesssim A^{O_n(1)}N^{1-\frac{3(n-1)}q-\beta}
}
using Holder's inequality, \eqref{one}, \eqref{critical}, and summing the geometric series, assuming additionally that $\max(-\frac2q,1-\frac{3(n-1)}q)<\beta\leq\min(\alpha_q,2-\frac{2(n-1)}q-)$. The ``low-high'' term is analogous. Thus \eqref{one} is proved, assuming that there exists a $\beta$ such that all the stated conditions on $\alpha$, $\beta$, and $n$ are satisfied. One easily checks that this is follows from the hypotheses given above.

It is straightforward to see that by taking the largest permissible $n$ satisfying the constraints of \eqref{one}, we can always make $q/n\in(1,2]$. Therefore we can apply \eqref{one} with $\alpha=0$ along with \eqref{localbernstein} and \eqref{critical} to find
\eqn{
\|P_Nu\|_{\X20{A_3/2}(B(x_1,A_4/2))}&\lesssim N^{\frac{3n}q-\frac32}\|P_Nu\|_{\X{q/n}0{A_3/2}(B_n)}+(A_4N)^{-50}A_4^{\frac32-\frac3q}N^{\alpha_q}A.
}
This implies, using \eqref{one}, H\"older's inequality, and \eqref{critical},
\eq{\label{l2}
\|P_Nu\|_{\X20{A_3/2}(B(x_1,A_4/2))}&\lesssim N^{-\frac12}A^{O(1)}
}
for $N\geq A_2^{-\frac12}$.

We bootstrap this estimate one more time to bring in factors of $A_1^{-1}$. The linear and global terms are estimated the same way as before so we neglect them and focus on the remaining parts. By the same calculation above, for $N\geq A_2^{-1/3}$,
\eqn{
\|P_Nu\|_{\X20{A_3/2}(B(x_1,A_4/4))}&\lesssim N^{-1}\|\tilde P_N(u\otimes u)\|_{\X20{A_3/4}(B(x_1,A_4/3))}.
}
Using the paraproduct decomposition and \eqref{localbernstein},
\eqn{
\|\tilde P_N(u\otimes u)\|_{\X20{A_3/2}(B(x_1,A_4/3))}\lesssim\sum_{N_1\lesssim N_2\sim N}\|P_{N_1}u\odot P_{N_2}u\|_{\X20{A_3/2}(B(x_1,A_4/2))}\\
+N^{3/4}\sum_{N_1\sim N_2\gtrsim N}\|P_{N_1}u\otimes P_{N_2}u\|_{L_t^\infty L_x^{4/3}(B(x_1,A_4/2))}.
}
By H\"older's inequality, \eqref{l2}, \eqref{pointwise}, and \eqref{linfty},
\begin{align*}
\sum_{N_1\lesssim N_2\sim N}\|P_{N_1}u\odot P_{N_2}u\|_{\X20{A_3/2}(B(x_1,A_4/2))}&\lesssim\sum_{N_1\lesssim N}A_1^{-1}A^{O(1)}N^{-1/2}\|P_{N_1}u\|_{\X\infty0{A_3/2}}\\
&\lesssim A^{O(1)}(A_2^{-1}+A_1^{-1}N^{1/2})
\end{align*}
and, using Young's inequality, interpolation, \eqref{l2}, \eqref{pointwise}, and \eqref{linfty}, 
\eqn{
\sum_{N_1\sim N_2\gtrsim N}\|P_{N_1}u\otimes P_{N_2}u\|_{\X{4/3}0{A_3/2}(B(x_1,A_4/2))}&\lesssim\sum_{N\lesssim N_1\leq A_2}A^{O(1)}N_1^{-1/4}A_1^{-1/2}\\
&\quad+\sum_{N_1\geq A_2}A^{O(1)}N_1^{-1/4}\\
&\lesssim A^{O(1)}(A_1^{-1/2}N^{-1/4}+A_2^{-1/4}).
}
Thus we conclude
\eq{\label{l2improved}
\|P_Nu\|_{\X20{A_3/2}(B(x_1,A_4/4))}&\lesssim A^{O(1)}((A_1N)^{-1/2}+(A_2N)^{-1/4}).
}

To reach the contradiction, Duhamel's formula and \eqref{lowerbound} give us
\eqn{
A_1^{-1}&\leq|P_1u(0,x_1)|\\
&\leq|e^{A_3\Delta/4}P_1u(-A_3/4)|(x_1)+\int_{-A_3/4}^0|e^{(t-t')\Delta}P_1\div\tilde P_1(u\otimes u)(t',x_1)|dt'.
}
For the first term, by \eqref{PNheat} and \eqref{critical},
\eqn{
|e^{A_3\Delta/4}P_1u(-A_3/4)|(x_1)\lesssim e^{-A_3/80}A
}
which is negligible compared to $A_1^{-1}$. Therefore
\eq{\label{conclusion}
\int_{-A_3/4}^0|e^{(t-t')\Delta}P_1\div\tilde P_1(u\otimes u)(t',x_1)|dt'\gtrsim A_1^{-1}.
}
By \eqref{localbernstein}, we have
\eqn{
|e^{(t-t')\Delta}P_1\div\tilde P_1(u\otimes u)(t',x_1)|\lesssim e^{-(t-t')/20}(\|\tilde P_1(u\otimes u)(t')\|_{L_x^\infty(B(x_1,A_4/8))}+A_4^{-50})
}
which admits the paraproduct decomposition
\eqn{
\|\tilde P_1(u\otimes u)\|_{\X\infty0{A_3/4}(B(x_1,A_4/8))}&\lesssim\sum_{N_1\lesssim N_2\sim 1}\|P_{N_1}u\odot P_{N_2}u\|_{\X\infty0{A_3/4}(B(x_1,A_4/4))}\\
&\quad+\sum_{1\lesssim N_1\sim N_2\leq A_2}\|P_{N_1}u\otimes P_{N_2}u\|_{\X{4/3}0{A_3/4}(B(x_1,A_4/4))}\\
&\quad+\sum_{N_1\sim N_2\geq A_2}\|P_{N_1}u\otimes P_{N_2}u\|_{\X10{A_3/4}(B(x_1,A_4/4))}.
}
We estimate each piece using H\"older's inequality, \eqref{linfty}, \eqref{pointwise}, \eqref{l2}, interpolation, and \eqref{l2improved}:
\eqn{
\sum_{N_1\lesssim N_2\sim1}\|P_{N_1}u\odot P_{N_2}u\|_{\X\infty0{A_3/4}(B(x_1,A_4/4))}&\lesssim\sum_{N_1\leq A_2^{-1},N_2\sim1}A^2N_1+\sum_{A_2^{-1}\leq N_1\lesssim N_2\sim1}A_1^{-2}N_1\\
&\lesssim A^2A_2^{-1}+A_1^{-2},
}
\eqn{
&\sum_{1\lesssim N_1\sim N_2\leq A_2}\|P_{N_1}u\otimes P_{N_2}u\|_{\X{4/3}0{A_3/4}(B(x_1,A_4/4))}\\
&\quad\quad\lesssim\sum_{1\lesssim N_1\leq A_2}A^{O(1)}((A_1N_1)^{-3/4}+(A_2N_1)^{-3/8})(A_1^{-1}N_1)^{1/2}\\
&\quad\quad\lesssim A^{O(1)}(A_1^{-5/4}+A_1^{-1/2}A_2^{-1/4}),
}
and
\eqn{
\sum_{N_1\sim N_2\geq A_2}\|P_{N_1}u\otimes P_{N_2}u\|_{\X10{A_3/4}(B(x_1,A_4/4))}&\lesssim\sum_{N_1\gtrsim A_2}A^{O(1)}N_1^{-1}\lesssim A^{O(1)}A_2^{-1}.
}
Comparing this upper bound with the lower bound \eqref{conclusion}, we reach the contradiction
\eqn{
A_1^{-1}\lesssim A^{O(1)}(A_2^{-1}+A_2^{-2}+A_1^{-5/4}+A_1^{-1/2}A_2^{-1/4}+A_2^{-1}).
}

Now we consider the $2<q<3$ case. By \eqref{lb}, \eqref{bernstein}, and \eqref{critical},
\eqn{
A_1^{-1}\leq|P_1u(0,x_1)|\leq r(x_1)^{-1+\frac2q}\|r^{1-\frac2q}P_1u(0)\|_{L_x^\infty}\lesssim r(x_1)^{-1+\frac2q}A.
}
Therefore we may assume $r(x_1)\lesssim A_2$ which will be useful at certain points to bound the heat propagator term when the power on the weight is larger than what we can handle using \eqref{bernstein}.

First, we show by induction on $n\geq1$ that
\eq{\label{three}
\|P_Nu\|_{\X q\alpha{T_n}(B_n)}&\lesssim N^{\alpha_q-\alpha}A^{O(1)}
}
if $N\geq A_2^{-1-\frac1n}$, $-\frac2q<\alpha<2-\frac3q$ and $\alpha\leq n(1-\frac2q)-\frac1q$, where now $T_n=(\frac12+\frac1{2n})A_3$ and $B_n=B(x_1,(\frac12+\frac1{2n})A_4)$. For $n=2$ and $t\in[-3A_3/4,0]$, first consider the linear term. Using the fact that $r(B_1)\leq A_4$, \eqref{PNheat}, and \eqref{critical},
\eqn{
\|P_Ne^{(t+2A_3)\Delta}u(-2A_3)\|_{\X q\alpha{A_3}(B_1)}&\lesssim A_4^{\alpha-\gamma}\|P_Ne^{(t+2A_3)\Delta}u(-2A_3)\|_{\X q\gamma{A_3}}\\
&\lesssim A_4^{\alpha-\gamma}N^{\alpha_q-\gamma}e^{-A_3N^2/20}A\\
&\lesssim(AN)^{-10}
}
upon taking, say, $\gamma=\min(\alpha_q,\alpha)$. Therefore, by H\"older's inequality, \eqref{localbernstein}, and \eqref{critical},
\eqn{
\|r^\alpha P_Nu(t)\|_{L_x^q(B_1)}&\lesssim(AN)^{-10}+\int_{-2A_3}^t\|r^\alpha P_Ne^{(t-t')\Delta}\mathbb P\div(u\otimes u)(t')\|_{L_x^q(\Rt)}dt'\\
&\lesssim(AN)^{-10}+\int_{-2A_3}^te^{-(t-t')N^2/20}N^{2+\alpha_q-\alpha}\|r^{2\alpha_q}u\otimes u\|_{L_x^{q/2}(\mathbb R^3)}dt'\\
&\lesssim N^{\alpha_q-\alpha}A^2
}
if $-\frac2q<\alpha\leq2-\frac5q$. Note that the lower bound $N\geq A_2^{-3/2}$ is essential to make the contribution of the linear term negligible. This completes the proof of the base case.

Next suppose we have the desired bound for some $n-1\geq2$. The linear term can be treated as in the previous case so we do not repeat the argument. Then, again by H\"older's inequality and \eqref{localbernstein}, for $t\in[-T_n,0]$,
\eqn{
\|r^\alpha P_Nu(t)\|_{L_x^q(B_n)}&\lesssim\|P_Ne^{(t+T_{n-1})\Delta}u(-T_{n-1})\|_{\X q\alpha{T_n}(B_n)}\\
&\quad+\int_{-T_{n-1}}^t\|r^\alpha P_Ne^{(t-t')\Delta}\mathbb P\div(u\otimes u)(t')\|_{L_x^q(B_n)}dt'\\
&\lesssim(A_4N)^{-10}+(A_4N)^{-50}A_3A_4^{\frac3q+\alpha-\gamma}N^{3-\frac3q-\gamma}\|u\otimes u\|_{\X{q/2}{2\alpha_q}{T_{n-1}}}\\
&\quad+N^{2+\beta-\alpha}\int_{-T_{n-1}}^te^{-(t-t')N^2/20}\bigg(\|r^{\beta+\alpha_q}P_{>N/100}u\otimes u\|_{L_x^{q/2}(B_{n-1})}\\
&\quad+\|r^{\beta+\alpha_q}P_{\leq N/100}u\otimes P_{>N/100}u)\|_{L_x^{q/2}(B_{n-1})}\bigg)dt'\\
&\lesssim(A_4N)^{-10}+N^{\beta-\alpha}\|P_{>N/100}u\|_{\X q\beta{T_{n-1}(B_{n-1})}}\|u\|_{\X q{\alpha_q}{T_{n-1}}}\\
&\lesssim N^{1-\frac3q-\alpha}A^{O(1)}
}
assuming $\gamma\leq\alpha$ and $-\frac2q<\gamma\leq2-\frac5q$ for the global Bernstein term, $-\frac2q<\alpha\leq\beta+\alpha_q+\frac1q$ and $\beta<1-\frac1q$ for the local Bernstein term, and $-\frac2q<\beta<2-\frac3q$ and $\beta\leq(n-1)(1-\frac2q)-\frac1q)$ for the inductive bound on $P_Nu$. One computes that such a $\beta$ and $\gamma$ exist under the stated conditions on $\alpha$ and $q$.

For fixed $q$, since $q>2$, the upper bound $n(1-\frac2q)-\frac1q$ becomes arbitrarily large by taking $n$ large so eventually the only constraint on $\alpha$ becomes $-\frac2q<\alpha<2-\frac3q$. Therefore we have
\eq{\label{cool}
\|P_Nu\|_{\X q\alpha {T/2}(B(x_1,1/2))}&\lesssim N^{\alpha_q-\alpha}A^{O(1)}
}
for all such $\alpha$ if $N\geq A_2^{-1}$. Now as in the $q>3$ case, we bootstrap this estimate one more time with \eqref{linfty} to bring in powers of $A_1^{-1}$. In the usual Duhamel formula for $P_Nu$ on $[-T_{n-1},0]$, we neglect the linear term and the global Bernstein term since they can be dealt with as above. Then by \eqref{localbernstein} and a paraproduct decomposition, for $t\in[-T_n,0]$,
\eqn{
\|P_Nu(t)\|_{L_x^q(B_n)}&\lesssim N^{-1}\sum_{N'\sim N}\|P_{N'}u\odot P_{\lesssim N}u\|_{\X q0{T_{n-1}}(B_{n-1})}\\
&\quad+N^{1-\frac2q}\sum_{N_1\sim N_2\gtrsim N}\|P_{N_1}u\otimes P_{N_2}u\|_{\X{3q/4}{2-3/q}{T_{n-1}}(B_{n-1})}.
}
For the first term, by H\"older's inequality, \eqref{cool} with $\alpha=0$, \eqref{pointwise}, and \eqref{linfty}, we have for $A_2^{-1/2}\leq N\leq A_2^{1/2}$
\eqn{
\sum_{N'\sim N}\|P_{N'}u\odot P_{\lesssim N}u\|_{\X q0{T_{n-1}}(B_{n-1})}&\lesssim N^{\alpha_q}A^{O(1)}\Bigg(\sum_{N_1\leq A_2^{-1}}AN_1+\sum_{A_2^{-1}\leq N_1\lesssim N}A_1^{-1}N_1\Bigg)\\
&\lesssim N^{1+\alpha_q}A^{O(1)}A_1^{-1}.
}
For the second term, by Young's inequality, the trivial interpolation inequality $$\|r^\beta f^2\|_{L^{3q/4}}\leq\|r^{3\beta/4}f\|_{L^q}^{4/3}\|f\|_{L^\infty}^{2/3},$$\eqref{cool} with $\alpha=\frac32-\frac9{4q}$, \eqref{linfty}, and \eqref{pointwise},
\eqn{
\sum_{N_1\sim N_2\gtrsim N}\|P_{N_1}u\otimes P_{N_2}u\|_{\X{3q/4}{2-\frac3q}{T_{n-1}}(B_{n-1})}&\lesssim\sum_{N\lesssim N_1\leq A_2}(N_1^{\alpha_q-(\frac32-\frac9{4q})}A^{O(1)})^{4/3}(A_1^{-1}N_1)^{2/3}\\
&\quad+\sum_{N_1\geq A_2}(N_1^{\alpha_q-(\frac32-\frac9{4q})}A^{O(1)})^{4/3}(AN_1)^{2/3}\\
&\lesssim A^{O(1)}(A_1^{-2/3}N^{-1/q}+A_2^{-1/q}).
}
Therefore, if $A_2^{-1/2}\leq N\leq A_2^{1/2}$,
\eq{\label{bootstrapped}
\|P_Nu\|_{\X q0{A_3/4}(B(x_1,A_4/4))}&\lesssim A^{O(1)}A_1^{-2/3}N^{\alpha_q}.
}

Now returning to \eqref{conclusion} and applying a paraproduct decomposition, H\"older's inequality, \eqref{bootstrapped}, \eqref{pointwise}, \eqref{linfty}, and \eqref{cool}, we have
\eqn{
A_1^{-1}&\lesssim\int_{-A_3/4}^0|e^{(t-t')\Delta}P_1\div(u\otimes u)(t',x_1)|dt'\\
&\lesssim\sum_{N'\sim1}\|P_{N'}u\odot P_{\lesssim 1}u\|_{\X q0{A_3/4}(B(x_1,A_4/4))}\\
&\quad+\sum_{N_1\sim N_2\gtrsim1}\|P_{N_1}u\otimes P_{N_1}u\|_{\X{q/2}0{A_3/4}(B(x_1,A_4/4))}\\
&\lesssim A^{O(1)}A_1^{-2/3}\left(\sum_{N_1\leq A_2^{-1}}AN_1+\sum_{A_2^{-1}\leq N_1\lesssim1}A_1^{-1}N_1\right)\\
&\quad+\sum_{1\lesssim N_1\leq A_2^{1/2}}A^{O(1)}A_1^{-4/3}N_1^{2\alpha_q}+\sum_{N_1\geq A_2^{1/2}}N_1^{2\alpha_q}A^{O(1)}\\
&\lesssim A^{O(1)}(A_1^{-2/3}A_2^{-1}+A_1^{-5/3}+A_1^{-4/3}+A_2^{\alpha_q})
}
which is the desired contradiction, recalling that $\alpha_q<0$ in this case.
\end{proof}

This proposition can be iterated exactly as in \cite{tao} to obtain the back propagation result we will need in the main argument.

\begin{prop}[{\cite[Proposition 3.1(v)]{tao}}]\label{iteratedbp}
Let $x_0\in\mathbb R^3$ and $N_0>0$ be such that
\eqn{
|P_{N_0}u(t_0,x_0)|\geq A_1^{-1}N_0.
}
Then for every $A_4N_0^{-2}\leq T_1\leq A_4^{-1}T$, there exists
\eqn{
(t_1,x_1)\in[t_0-T_1,t_0-A_3^{-1}T_1]\times\Rt
}
and
\eqn{
N_1=A_3^{O(1)}T_1^{-\frac12}
}
such that
\eqn{
x_1=x_0+O(A_4^{O(1)}T_1^{\frac12})
}
and
\eqn{
|P_{N_1}u(t_1,x_1)|\geq A_1^{-1}N_1.
}
\end{prop}

We do not repeat the proof from \cite{tao} because it would proceed in the exact same manner now that we have all the building blocks: the back propagation proposition (Proposition \ref{backpropagationprop}), the pointwise bounds for the frequency-localized vector fields (Lemma \ref{pointwiselemma}), and the total bounded speed property (Proposition \ref{tbsprop}).

\subsection{Regularity away from the axis}

In the proof of Proposition \ref{main}, it will be necessary to propagate a concentration of the solution forward in time using the Carleman estimate from Proposition \ref{backwardunique}. Analogously to Proposition \ref{epochs}, Proposition \ref{regularityprop} will provide us with the pointwise control required for the Carleman inequality. In fact this is the main reason why the final theorems have improved bounds over those in \cite{tao}---we are able to find spacetime regions where $u$ has good estimates such that the distance from the axis is roughly on the order of $AT^{1/2}$. This should be compared to the $u\in L_t^\infty L_x^3$ case where, because bootstrapping is necessary, the annulus (or cylindrical shell) of regularity may have a distance from $x=0$ (or the $x_3$-axis) as large as $\sim e^AT^{1/2}$. See Section \ref{intro} for the heuristics that suggest this result.

\begin{prop}\label{regularityprop}
Suppose $u$ is as in Proposition \ref{epochs}. If $T'\in[0,T/2]$ and $R\geq(T')^{1/2}$, then in the region
\eqn{
\Omega=\{(t,x)\in[t_0-T',t_0]\times\mathbb R^3:r\geq R\},
}
we have
\eqn{
\|\grad^ju\|_{L_{t,x}^\infty(\Omega)}\lesssim (T')^{-\frac{j+1}2}\left(\frac{R^2}{T'}\right)^{-1/O(1)}A^{O(1)}
}
and
\eqn{
\|\grad^j\omega\|_{L_{t,x}^\infty(\Omega)}\lesssim (T')^{-\frac{j+2}2}\left(\frac{R^2}{T'}\right)^{-1/O(1)}A^{O(1)}
}
for $j=0,1$.
\end{prop}

\begin{proof}

Let us shift time and rescale to achieve $t_0=0$ and $T'=1$. First we note that by \eqref{bernstein}, \eqref{ulin-bound}, and \eqref{critical}, if $q>3$ and $f$ is either $\ulin_n$ or $\unlin_n$, we have
\eqn{
R^{1-\frac3q}\|P_Nf\|_{L_{t,x}^\infty([-1/2,0]\times\{r\geq R/10\})}\lesssim\|r^{1-\frac3q}P_Nf\|_{L_{t,x}^\infty([-1/2,0]\times\Rt)}&\lesssim N^{\frac3q}A^{O_n(1)}
}
and if instead $q\leq3$ and $u$ is axisymmetric, then
\eqn{
R^{1-\frac2q}\|P_Nf\|_{L_{t,x}^\infty([-1/2,0]\times\{r\geq R/10\})}\lesssim\|r^{1-\frac2q}P_Nf\|_{L_{t,x}^\infty([-1/2,0]\times\Rt)}&\lesssim N^{\frac2q}A^{O_n(1)}.
}
Let us therefore define $\thing_q=\frac3q$ in the former case and $\thing_q=\frac2q$ in the latter so that we always have
\eq{\label{help}
\|P_N\ulin_n\|_{L_{t,x}^\infty([-1/2,0]\times\{r\geq R/10\})},\:\|P_N\unlin_n\|_{L_{t,x}^\infty([-1/2,0]\times\{r\geq R/10\})}&\lesssim N^{\thing_q}R^{-1+\thing_q}A^{O_n(1)}.
}
Importantly, in either case, $\thing_q<1$. The point is that while staying uniformly away from the $x_3$-axis, this is a subcritical estimate and we can iteratively improve it with Duhamel's principle. Let us begin with a straightforward application of \eqref{PNheat}, H\"older's inequality, and \eqref{critical} to obtain, for $t\in[-1/2,0]$,
\eqn{
\|r^{2\alpha_q}P_N\unlin_1(t)\|_{L_x^{q/2}(\mathbb R^3)}&\leq\int_{-1/2}^t\|r^{2\alpha_q}P_Ne^{(t-t')\Delta}\mathbb P\div(u\otimes u)(t')\|_{L_x^{3/2}(\Rt)}dt'\\
&\lesssim\int_{-1/2}^te^{-(t-t')N^2/20}N\|r^{2\alpha_q}u\otimes u\|_{L_x^{q/2}(\Rt)}dt'\\
&\lesssim N^{-1}A^2.
}
Next, we have
\eqn{
P_N\unlin_n(t)=\int_{-1/2}^tP_Ne^{(t-t')\Delta}\div\tilde P_N(\unlin_{n-1}\otimes\unlin_{n-1}+2\unlin_{n-1}\odot\ulin_{n-1})dt'.
}
By \eqref{localbernstein}, for $t\in[-1/2,0]$, we have
\eqn{
&\|r^{2\alpha_q}P_N\unlin_n(t)\|_{L_x^{q/2}(r\geq(\frac12-2^{-n})R)}\\
&\quad\lesssim_n N^{-1}\|r^{2\alpha_q}(Y_1+Y_2+Y_3+Y_4+Y_5)\|_{L_x^{q/2}(r\geq(\frac12-2^{-(n-1)})R)}\\
&\quad\quad+(NR)^{-50}N^{-1}\|r^{2\alpha_q}(\unlin_{n-1}\otimes\unlin_{n-1}+2\unlin_{n-1}\odot\ulin_{n-1})\|_{L_x^{q/2}(\mathbb R^3)}
}
where we decompose $P_N(\unlin_{n-1}\otimes\unlin_{n-1}+2\unlin_{n-1}\odot\ulin_{n-1})$ with the paraproducts
\eqn{
Y_1&=2\sum_{N'\sim N}P_{N'}\unlin_{n-1}\odot P_{\leq N/100}\unlin_{n-1}\\
Y_2&=\sum_{N_1\sim N_2\gtrsim N}P_{N_1}\unlin_{n-1}\otimes P_{N_2}\unlin_{n-1}\\
Y_3&=\sum_{N_1\sim N_2\gtrsim N}P_{N_1}\ulin_{n-1}\otimes P_{N_2}\unlin_{n-1}\\
Y_4&=2\sum_{N'\sim N}P_{N'}\ulin_{n-1}\odot P_{\leq N/100}\unlin_{n-1}\\
Y_5&=2\sum_{N'\sim N}P_{\leq N/100}\ulin_{n-1}\odot P_{N'}\unlin_{n-1}
}
By H\"older's inequality, \eqref{ulin-bound}, and \eqref{critical}, the global Bernstein term is bounded by $(NR)^{-50}N^{-1}A^{O_n(1)}$. Let $\Omega_n=[-1/2,0]\times\{r\geq(\frac12-2^{-n})R\}$. Assuming $N\gtrsim_{n,c}R^c$, where $c>0$ is a small constant depending on $q$, by H\"older's inequality, \eqref{bernstein}, \eqref{ulin-bound}, \eqref{critical}, \eqref{help}, \eqref{ulinPNq}, and \eqref{ulinPNinfty},
\eqn{
\|r^{2\alpha_q}Y_1\|_{L_t^\infty L_x^{q/2}(\Omega_{n-1})}&\lesssim\sum_{N'\sim N}\|r^{2\alpha_q}P_{N'}\unlin_{n-1}\|_{L_t^\infty L_x^{q/2}(\Omega_{n-1})}\\
&\quad\quad\times\sum_{N'\lesssim n}A^{O_n(1)}N'\max(N'R,1)^{-1+\gamma_q}\\
&\lesssim A^{O_n(1)}N^{\thing_q}R^{-1+\thing_q}\sum_{N'\sim N}\|r^{2\alpha_q}P_{N'}\unlin_{n-1}\|_{L_t^\infty L_x^{q/2}(\Omega_{n-1})},
}
\eqn{
\|r^{2\alpha_q}Y_2\|_{L_t^\infty L_x^{q/2}(\Omega_{n-1})}&\lesssim\sum_{N_1\sim N_2\gtrsim N}A^{O_n(1)}N_1^{\thing_q}R^{-1+\thing_q}\|r^{2\alpha_q}P_{N_2}\unlin_{n-1}\|_{L_t^\infty L_x^{q/2}(\Omega_{n-1})},
}
\eqn{
\|r^{2\alpha_q}Y_3\|_{L_t^\infty L_x^{q/2}(\Omega_{n-1})}&\lesssim\sum_{N_1\sim N_2\gtrsim N}e^{-N_1^2/O_n(1)}N_1
A^{O_n(1)}\|r^{2\alpha_q}P_{N_2}\unlin_{n-1}\|_{L_t^\infty L_x^{q/2}(\Omega_{n-1})},
}
\eqn{
\|r^{2\alpha_q}Y_4\|_{L_t^\infty L_x^{q/2}(\Omega_{n-1})}&\lesssim\sum_{N'\sim N}\|P_{N'}\ulin_{n-1}\|_{\X q{\alpha_q}1}\|P_{\lesssim N}\unlin_{n-1}\|_{\X q{\alpha_q}1}\\
&\lesssim e^{-N^2/O_n(1)}A^{O_n(1)},
}
and
\eqn{
\|r^{2\alpha_q}Y_5\|_{L_t^\infty L_x^{q/2}(\Omega_{n-1})}&\lesssim\sum_{N'\lesssim N}e^{-(N')^2/O_n(1)}N'A^{O_n(1)}\sum_{N'\sim N}\|r^{2\alpha_q}P_{N'}\unlin_{n-1}\|_{L_t^\infty L_x^{q/2}(\Omega_{n-1})}\\
&\lesssim A^{O_n(1)}\sum_{N'\sim N}\|r^{2\alpha_q}P_{N'}\unlin_{n-1}\|_{L_t^\infty L_x^{q/2}(\Omega_{n-1})}.
}
In total,
\eqn{
\|r^{2\alpha_q}P_N\unlin_n\|_{L_t^\infty L_x^{q/2}(\Omega_n)}&\lesssim A^{O_n(1)}((NR)^{1-\thing_q}+N)^{-1}\sum_{N'\sim N}\|r^{2\alpha_q}P_N\unlin_{n-1}\|_{L_t^\infty L_x^{q/2}(\Omega_{n-1})}\\
&\quad+R^{-1+\thing_q}N^{-1}A^{O_n(1)}\sum_{N_1\gtrsim N}N_1^{\thing_q}\|r^{2\alpha_q}P_{N_1}\unlin_{n-1}\|_{L_t^\infty L_x^{q/2}(\Omega_{n-1})}\\
&\quad+N^{-1}A^{O_n(1)}e^{-N^2/O_n(1)}.
}
Iteratively applying this, we find
\eq{\label{fast}
\|r^{2\alpha_q}P_N\unlin_n\|_{\X{q/2}{2\alpha_q}1(r\geq R/2)}&\lesssim_nA^{O_n(1)}N^{-1}\min((NR)^{1-\thing_q},N)^{-n+1},
}
noting that the assumption $N\gtrsim_{n,c}R^c$ implies
\eqn{
e^{-N^2/O_n(1)}\lesssim((NR)^{1-\gamma_q},N)^{-n+1}.
}
In order to make use of \eqref{fast}, we take the Littlewood-Paley decomposition of $\grad^j\unlin_n$ and apply \eqref{bernstein} and \eqref{localbernstein} to find
\eqn{
\|\grad^j\unlin_n\|_{L_{t,x}^\infty(\Omega)}&\lesssim\|P_{\leq R^{-1}}\grad^j\unlin_n\|_{\X\infty01}+\sum_{R{-1}<N\lesssim_{n,c}R^c}N^j\|P_N\unlin_n\|_{L_{t,x}^\infty(\Omega)}\\
&+\sum_{N\gtrsim_{n,c}R^c}\bigg(N^{2+j}\|P_N\unlin_n\|_{\X{q/2}{2\alpha_q}1(r\geq R/2)}+(NR)^{-50}N^{1+j}\|P_N\unlin_n\|_{\X q{\alpha_q}1}\bigg).
}
Thanks to \eqref{bernstein}, \eqref{ulin-bound}, and \eqref{critical}, the first term is bounded by $R^{-1-j}A^{O_n(1)}$. The global Bernstein term is estimated the same way, and summing the geometric series, we obtain the bound $A^{O_n(1)}R^{-40}$. For the intermediate frequency term, we apply \eqref{cool} and sum the geometric series to find
\eqn{
\sum_{R^{-1}<N\lesssim_{n,c}R^c}N^j\|P_N\unlin_n\|_{L_{t,x}^\infty(\Omega)}&\lesssim_{n,c}A^{O_n(1)}R^{-1+\thing_q+(\thing_q+j)c}.
}
For the high frequency term, we have to split the sum once again depending on how the minimum is attained in \eqref{fast}.
\eqn{
\sum_{N\gtrsim_{n,c}R^c}N^{2+j}\|P_N\unlin_n\|_{\X{q/2}{2\alpha_q}1(r\geq R/2)}&\lesssim\sum_{R^c\lesssim_{n,c}N\lesssim R^{\frac1{\thing_q}-1}}A^{O_n(1)}N^{j-n+2}\\
&\quad+\sum_{N\gtrsim R^{\frac1{\thing_q}-1}}A^{O_n(1)}N^{1+j}(NR)^{-(1-\thing_q)(n-1)}\\
&\lesssim_{n,c}A^{O_n(1)}R^{\min(\frac1{\thing_q}-1,c)(j-n+2)}
}
where $j=0,1,2,3$, assuming $n>10/(1-\thing_q)$ in order to make the series summable. By taking $n$ and $c^{-1}$ sufficiently large depending on $q$, all the powers on $R$ can be made uniformly negative, that is to say
\eqn{
\|\grad^j\unlin_n\|_{L_{t,x}^\infty(\Omega)}&\lesssim A^{O_n(1)}R^{-1/O_n(1)}.
}
Moreover, by essentially the same argument we used for \eqref{help}, we have
\eqn{
\|\grad^j\ulin_n\|_{L^\infty_{t,x}(\Omega)}&\lesssim A^{O_n(1)}R^{-(1-\thing_q)}.
}
Since $u=\ulin_n+\unlin_n$ and $\gamma_q<1$, this proves the estimates for $\grad^ju$.

The estimate for the vorticity follows by analogously estimating $\|r^{2\alpha_q}P_N\curl\unlin_n\|_{L_x^{q/2}(r\gtrsim R)}$ and $\|\grad^j\curl\ulin_n\|_{L_x^\infty(r\gtrsim R)}$ and concluding in the same manner.
\end{proof}

\section{Carleman estimates}\label{carleman}

We quote from \cite[Lemma 4.1]{tao} the general Carleman inequality for the backward heat operator $L=\dd_t+\Delta$ from which Carleman inequalities for specific domains and weight functions can be derived. Note that it is conventional to work with the backward heat operator even though we intend to apply these estimates to the forward heat equation.
\begin{lemma}
Let $[t_1,t_2]$ be a time interval and $u:C_c^\infty([t_1,t_2]\times\mathbb R^d\to\mathbb R^m)$ solve the backwards heat equation
\eqn{
Lu=f.
}
Fix a smooth weight function $g:[t_1,t_2]\times\mathbb R^d\to\mathbb R$ and define
\eqn{
F=\dd_tg-\Delta g-|\grad g|^2.
}
Then we have
\begin{equation}\label{generalcarleman}
\begin{aligned}
&\int_{t_1}^{t_2}\int_{\mathbb R^d}\left(\frac12(LF)|u|^2+2D^2g(\grad u,\grad u)\right)e^gdxdt\\
&\quad\leq\frac12\int_{t_1}^{t_2}\int_{\mathbb R^d}|Lu|^2e^gdxdt+\int_{\mathbb R^d}\left(|\grad u|^2+\frac12F|u|^2\right)e^gdx\Big|_{t=t_1}^{t=t_2}.
\end{aligned}
\end{equation}
\end{lemma}

Our first application of this lemma is to a Carleman estimate resembling the one used to prove backward uniqueness for the heat operator in \cite{ess} and the quantitative analog appearing in \cite{tao}. Unfortunately that estimate relies on the differential inequality \eqref{differentialinequality} holding in an annular region (or, in the qualitative case, the complement of a ball), which cannot possibly be contained in the cylindrical regions of regularity provided by Proposition \ref{regularityprop}. Thus we prove a variant that is suited to this geometry.

In this section only, since we give the result in the general setting $\mathbb R^{d_1+d_2}$ (where the last $d_2$ coordinates correspond to the ``axis''), we extend the definition of $r$ and define $|z|$ to be
\eqn{
r:=\sqrt{x_1^2+\cdots+x_{d_1}^2},\quad |z|:=\sqrt{x_{d_1+1}^2+\cdots x_{d_1+d_2}^2}.
}
The regions $\shell{r_-}{r_+}$, etc.\ are defined in the same way as before but in terms of the generalized $r$ and $|z|$ coordinates, where naturally $|z|$ replaces $|x_3|$.

Proposition \ref{backwardunique} generalizes a quantitative Carleman inequality from \cite{tao} which corresponds to the case $d_1=3$, $d_2=0$. In the paper at hand, we will be using the case $d_1=2$, $d_2=1$.

\begin{prop}[Backward uniqueness Carleman estimate]\label{backwardunique}
Let $d_1\geq1$, $d_2\geq0$, $T>0$, $0<r_-<r_+$, and $\mathcal C$ denote the spacetime region
\eqn{
\mathcal C=\{(t,x)\in\mathbb R\times\mathbb R^{d_1+d_2}:t\in[0,T],\,r_-\leq r\leq r_+,\,|z|\leq r_+\}.
}
Let $u:\mathcal C\to\mathbb R$ be a smooth function obeying the differential inequality
\eq{\label{differentialinequality}
|Lu|\leq\frac1{C_0T}|u|+\frac1{(C_0T)^{1/2}}|\grad u|
}
on $\mathcal C$. Assume the inequality
\eqn{
r_-^2\geq4C_0T.
}
Then one has
\eqn{
\int_0^{T/4}\int_{\shelltrunc{10r_-}{\frac{r_+}2}{\frac{r_+}2}}(T^{-1}|u|^2+|\grad u|^2)dxdt&\lesssim C_0e^{-\frac{r_-r_+}{4C_0T}}(X+e^{\frac{2r_+^2}{C_0T}}Y)
}
where
\eqn{
X=\int\int_{\mathcal C}e^{2|x|^2/C_0T}(T^{-1}|u(t,x)|^2+|\grad u(t,x)|^2)dxdt
}
and
\eqn{
Y=\int_{\shelltrunc{r_-}{r_+}{r_+}}|u(0,x)|^2dx.
}
\end{prop}

\begin{proof}
We may assume $r_+\geq20r_-$. The pigeonhole principle implies the existence of a $T_0\in[T/2,T]$ such that
\eq{
\int_{\shelltrunc{r_-}{r_+}{r_+}}e^{2|x|^2/C_0T}(T^{-1}|u(T_0,x)|^2+|\grad u(T_0,x)|^2)dx\lesssim T^{-1}X.
}
With the weight
\eqn{
g(x,t)=\frac{r_+(T_0-t)}{2C_0T^2}r+\frac1{C_0T}|x|^2,
}
we apply the general Carleman inequality to $\psi u$, where $\psi$ is a smooth spatial cutoff supported in $\shelltrunc{r_-}{r_+}{r_+}$ that equals $1$ in $\shelltrunc{2r_-}{r_+/2}{r_+/2}$ and obeys $|\grad^j\psi(x)|\lesssim r_-^{-j}$ for $j=0,1,2$. Since the function $r$ is convex, we have
\eqn{
D^2g\geq\frac2{C_0T}Id
}
as quadratic forms. With $F=\dd_tg-\Delta g-|\grad g|^2$, we compute
\eqn{
F&=-\frac{r_+}{2C_0T^2}r-\frac{r_+(T_0-t)}{2C_0T^2}\frac{d_1-1}r-\frac{2(d_1+d_2)}{C_0T}-\frac{r_+^2(T_0-t)^2}{4C_0^2T^4}-\frac4{C_0^2T^2}|x|^2-\frac{2r_+(T_0-t)}{C_0^2T^3}r\\
&\leq0.
}
It follows that
\eqn{
LF=\frac{r_+^2(T_0-t)}{2C_0^2T^4}+\frac{2r_+}{C_0^2T^3}r-\frac{r_+(T_0-t)}{2C_0T^2}\frac{(d_1-1)(3-d_1)}{r^3}\\
-\frac{8(d_1+d_2)}{C_0^2T^2}-\frac{2r_+(T_0-t)}{C_0^2T^3}\frac{d_1-1}r.
}
By using the bounds $2(C_0T)^{1/2}\leq r_-\leq r\leq r_+$, one finds that
\eqn{
\frac{r_+(T_0-t)}{2C_0T^2}\frac{(d_1-1)(3-d_1)}{r^3}+\frac{8(d_1+d_2)}{C_0^2T^2}+\frac{2r_+(T_0-t)}{C_0^2T^3}\frac{d_1-1}r\leq\frac{3(d_1+d_2)r_+}{C_0^3T^3}r.
}
Therefore, letting $C_0\geq 3(d_1+d_2)$,
\eqn{
LF\geq\frac{r_+r}{C_0^2T^3}\geq\frac{4}{C_0T^2}.
}
Putting this information into the general inequality \eqref{generalcarleman}, we have
\eqn{
&\int_0^{T_0}\int_{\shelltrunc{2r_-}{\frac{r_+}2}{\frac{r_+}2}}\left(\frac{2}{C_0T^2}|u|^2+\frac4{C_0T}|\grad u|^2\right)e^gdxdt\\
&\quad\leq\frac12\int_0^{T_0}\int_{\mathbb R^{d_1+d_2}}|L(\psi u)|^2e^gdxdt+\int_{\mathbb R^{d_1+d_2}}|\grad(\psi u)(T_0,x)|^2e^{g(T_0,x)}dx\\
&\quad\quad+\frac12\int_{\mathbb R^{d_1+d_2}}|F(0,x)||\psi u(0,x)|^2e^{g(0,x)}dx.
}
In the region $\shelltrunc{2r_-}{\frac{r_+}2}{\frac{r_+}2}$, $\psi$ is identically $1$ so thanks to the pointwise bound on $Lu$, this part of the integral in the first term on the right-hand side can be absorbed into the left-hand side. Moreover, throughout all of $\mathcal C$, using the bounds on $\grad^j\psi$ and $r_-$,
\eqn{
|L(\psi u)|^2=|\psi Lu+2\grad\psi\cdot\grad u+(\Delta\psi)u|^2&\lesssim(C_0T)^{-2}|u|^2+(C_0T)^{-1}|\grad u|^2.
}
Similarly,
\eqn{
|\grad(\psi u)|^2=|\psi\grad u+\grad\psi u|^2\lesssim|\grad u|^2+(C_0T)^{-1}|u|^2.
}
By limiting the time interval for the integral on the left-hand side to $[0,T/4]$ and the $r$ interval to $[10r_-,r_+/2]$, we find that on this region of integration
\eqn{
g(x,t)\geq\frac{5r_-r_+}{4C_0T}.
}
Therefore
\eqn{
&e^{\frac{5r_-r_+}{4C_0T}}\int_0^{T/4}\int_{\shelltrunc{10r_-}{\frac{r_+}2}{\frac{r_+}2}}\left(\frac1{C_0T^2}|u|^2+\frac1{C_0T}|\grad u|^2\right)dxdt\\
&\quad\lesssim\int_0^{T_0}\int_{\shelltrunc{r_-}{2r_-}{\frac{r_+}2}\cup\shelltrunc{\frac{r_+}2}{r_+}{\frac{r_+}2}\cup\shelltrunc{r_-}{r_+}{\frac{r_+}2,r_+}}\left(\frac1{(C_0T)^2}|u|^2+\frac1{C_0T}|\grad u|^2\right)e^gdxdt\\
&\quad\quad+\int_{\shelltrunc{r_-}{r_+}{r_+}}\Big(|\grad u(T_0,x)|^2+\frac1{C_0T}|u(T_0,x)|^2\Big)e^{g(T_0,x)}\\
&\quad\quad+\int_{\shelltrunc{r_-}{r_+}{r_+}}|F(0,x)||u(0,x)|^2e^{g(0,x)}dx.
}
Consider the first term on the right-hand side. Within the region of integration, we have
\eqn{
g(x,t)-\frac{2|x|^2}{C_0T}-\frac{5r_-r_+}{4C_0T}&=\frac{r_+(T_0-t)}{2C_0T^2}r-\frac1{C_0T}|x|^2-\frac{5r_-r_+}{4C_0T}\leq-\frac{r_+r_-}{4C_0T}.
}
Indeed, in $\shell{r_-}{2r_-}\cup\shell{\frac{r_+}2}{r_+}$, this is maximized at $r=2r_-$, $|z|=0$ where the given upper bound holds. In $\{|z|\in[r_+/2,r_+]\}$, the quantity is clearly largest when $|z|=r_+/2$, so we have the upper bound
\eqn{
\frac{r_+}{2C_0T}r-\frac1{C_0T}\Big(r^2+\frac{r_+^2}4\Big)-\frac{5r_-r_+}{4C_0T}
}
which is largest when $r=r_+/4$, yielding an upper bound of $-\frac{3r_+^2}{16C_0T}\leq-\frac{r_+r_-}{4C_0T}$.

In conclusion, after dividing both sides of the inequality by $e^{\frac{5r_-r_+}{4C_0T}}$, the first term on the right-hand side has a weight bounded by $e^{\frac{2|x|^2}{C_0T}-\frac{r_+r_-}{4C_0T}}$ so the whole term can be absorbed into $e^{-\frac{r_+r_-}{4C_0T}}X/T$. Similarly, $e^{g(T_0,x)}=e^{\frac{|x|^2}{C_0T}}$ so by the definition of $T_0$, the second term on the right has the same upper bound. Thus we have
\eqn{
&\int_0^{T/4}\int_{\shelltrunc{10r_-}{\frac{r_+}2}{r_+/2}}\left(\frac1{C_0T^2}|u|^2+\frac1{C_0T}|\grad u|^2\right)dxdt\\
&\quad\lesssim e^{-\frac{r_+r_-}{4C_0T}}\left(T^{-1}X+\int_{\shelltrunc{r_-}{r_+}{r_+}}|F(0,x)||u(0,x)|^2e^{g(0,x)}dx\right).
}
To conclude, we easily have
\eqn{
|F(0,x)|\leq\frac{r_+^2}{C_0T^2}
}
and
\eqn{
e^{g(0)}\leq e^{\frac{3r_+^2}{2C_0T}}
}
when $x$ is restricted to $\shelltrunc{r_-}{r_+}{r_+}$. Therefore
\eqn{
|F(0,x)|e^{g(0)}\leq e^{\frac{2r_+^2}{C_0T}}T^{-1}
}
which completes the proof.
\end{proof}

The next Carleman inequality we quote directly from \cite{tao}.

\begin{prop}[Unique continuation Carleman inequality]\label{uniquecont}
Define the cylindrical spacetime region
\eqn{
\mathcal C=\{(t,x)\in\mathbb R\times\Rt:t\in[0,T],\,|x|\leq\rho\}.
}
Let $u:\mathcal C\to\Rt$ be a smooth function obeying \eqref{differentialinequality} on $\mathcal C$. Assume
\eqn{
\rho^2\geq4000T.
}
Then for any
\eqn{
0<t_1\leq t_0\leq\frac T{1000},
}
one has
\eqn{
\int_{t_0}^{2t_0}\int_{|x|\leq\frac \rho2}(T^{-1}|u|^2+|\grad u|^2)e^{-|x|^2/4t}dxdt\lesssim e^{-\frac{\rho^2}{500t_0}}X+t_0^{3/2}(et_0/t_1)^{O(\rho^2/t_0)}Y
}
where
\eqn{
X=\int_0^T\int_{|x|\leq \rho}(T^{-1}|u|^2+|\grad u|^2)dxdt
}
and
\eqn{
Y=\int_{|x|\leq \rho}|u(0,x)|^2t_1^{-3/2}e^{-|x|^2/4t_1}dx.
}
\end{prop}

\section{Proof of Theorems \ref{regtheorem} and \ref{blowuptheorem}}\label{theorems}

Theorems \ref{regtheorem} and \ref{blowuptheorem} will follow without much difficulty from the following proposition.

\begin{prop}\label{main}
Let $u$ be as in Proposition \ref{epochs}, with $A\geq C_0$. Suppose that there exist $x_0\in\Rt$ and $N_0>0$ such that
\eqn{
|P_{N_0}u(t_0,x_0)|\geq A_1^{-1}N_0.
}
Then
\eqn{
TN_0^2\leq\exp(\exp(A_6^{O(1)})).
}
\end{prop}

\begin{proof}
By translating the solution, we may assume $t_0=(x_0)_3=0$. Note that we can shift to make the third component of $x_0$ vanish but not the first two as the norm $\X q\alpha T$ is not shift-invariant in those directions.

As shown in \cite{tao}, by propagating the concentration of $|P_{N_0}u|$ backward in time using Proposition \ref{iteratedbp}, converting it into a lower bound on the vorticity, and applying Proposition \ref{uniquecont} within an epoch of regularity provided by Proposition \ref{epochs}, one can deduce the following: for every $T_1\in[A_4N_0^{-2},A_4^{-1}T]$, and every $x_*\in\mathbb R^3$ with $|x_*-x_0|\geq A_4T_1^{1/2}$, we have the concentration
\eqn{
\int_{B(x_*,|x_*|/2)}|\omega(t,x)|^2dx\gtrsim\exp(-O(A_5^3|x_*|^2/T_1))T_1^{-1/2}
}
for all $t\in I$ where $I\subset[-T_1,-A_3^{-O(1)}T_1]$ is a time interval with $|I|=A_3^{-O(1)}T_1$. (We do not repeat these arguments because they hold in our setting without modification.)

In order to make use of this lower bound, we need some control on the location of $x_0$. As in the proof of Proposition \ref{regularityprop}, letting $\thing_q=\frac2q$ in the case where $q\in(2,3]$ with $u$ axisymmetric and $\thing_q=\frac3q$ in the case where $3<q<\infty$, we compute using \eqref{bernstein} and \eqref{critical}
\eqn{
A_1^{-1}N_0\leq|P_{N_0}u(t_0,x_0)|&\leq r(x_0)^{-1+\thing_q}\|r^{1-\thing_q}P_{N_0}u(t_0)\|_{L_x^\infty(\Rt)}\\
&\lesssim r(x_0)^{-1+\thing_q}N_0^{\thing_q}A
}
and therefore
\eqn{
|x_0|\lesssim A_1^{O(1)}N_0^{-1}\leq A_3^{-1}T_1^{1/2}.
}
Thus we deduce the lower bound
\eq{\label{conc1}
\int_{-T_1}^{-A_4^{-1}T_1}\int_{\shelltrunc{R}{10R}{10R}}|\omega(t,x)|^2dxdt\gtrsim\exp(-O(A_5^3R^2/T_1))T_1^{1/2}
}
for any $R\geq A_4^2T_1^{1/2}$, since the domain of this integral necessarily contains a ball $B(x_*,|x_*|/2)$ such that $2A_4T_1^{1/2}\leq|x_*|\lesssim R$. In order to propagate this concentration forward in time, we need some regularity on $u$ and $\omega$. For any $T_2\in[A_4^2N_0^{-2},A_4^{-1}T]$, by Proposition \ref{regularityprop}, we have
\eq{\label{regularity}
|\grad^ju(x,t)|\leq T_2^{-\frac{1+j}2}A_5^{-1/O(1)},\quad|\grad^j\omega(x,t)|\leq T_2^{-\frac{2+j}2}A_5^{-1/O(1)}
}
for $j=0,1$ and all $(t,x)\in[-T_2,0]\times\{r\geq A_5T_2^{1/2}\}$. This allows us to apply Proposition \ref{backwardunique} on $[0,T_2/C_0]$ with $r_-=A_5^2T_2^{1/2}$, $r_+=A_6T_2^{1/2}$, and $u$ replaced by the function
\eqn{
(t,x)\mapsto\omega(-t,x).
}
The vorticity equation
\eq{\label{vorticity}
\dd_t\omega-\Delta\omega=\omega\cdot\grad u-u\cdot\grad\omega
}
along with the coefficient bounds for the right-hand side coming from \eqref{regularity} imply \eqref{differentialinequality}. Letting
\eqn{
X=\int_{-T_2/C_0}^0\int_{\shelltrunc{A_5^2T_2^{1/2}}{A_6T_2^{1/2}}{A_6T_2^{1/2}}}e^{2|x|^2/T_2}(T_2^{-1}|\omega|^2+|\grad\omega|^2)dxdt,
}
\eqn{
Y=\int_{\shelltrunc{A_5^2T_2^{1/2}}{A_6T_2^{1/2}}{A_6T_2^{1/2}}}|\omega(0,x)|^2dx,
}
and
\eqn{
Z=T_2^{-1}\int_{-T_2/4C_0}^0\int_{\shelltrunc{10A_5^2T_2^{1/2}}{A_6T_2^{1/2}/2}{A_6T_2^{1/2}/2}}|\omega(x,t)|^2dxdt,
}
the Carleman estimate gives
\eqn{
Z\lesssim e^{-A_5^2A_6/4}X+e^{2A_6^2}Y.
}
From \eqref{conc1}, we have
\eqn{
Z\gtrsim T_2^{-1/2}e^{-O(A_5^5)}.
}
Thus either
\eq{\label{casewon}
X\gtrsim T_2^{-1/2}e^{A_6}
}
or
\eq{\label{casetoo}
Y\gtrsim T_2^{-1/2}e^{-3A_6^2}.
}
First let us assume that the concentration comes from \eqref{casewon} which is the harder case. Then
\eqn{
\int_{-T_2/C_0}^0\int_{\shelltrunc{A_5^2T_2^{1/2}}{A_6T_2^{1/2}}{A_6T_2^{1/2}}}e^{2|x|^2/T_2}(T_2^{-1}|\omega|^2+|\grad\omega|^2)dxdt\gtrsim T_2^{-1/2}e^{A_6}.
}
By \eqref{regularity}, the integrand is bounded by $T_2^{-3}e^{4A_6^2}$ and the region of integration in the $(t,x_1,x_2)$ variables has volume $O(A_6^2T_2^2)$. Therefore the range of $x_3$ in the integral can be narrowed without changing the inequality to
\eqn{
\int_{-T_2/C_0}^0\int_{\shelltrunc{A_5^2T_2^{1/2}}{A_6T_2^{1/2}}{e^{-A_6^3}T_2^{1/2},A_6T_2^{1/2}}}e^{2|x|^2/T_2}(T_2^{-1}|\omega|^2+|\grad\omega|^2)dxdt\\\gtrsim T_2^{-1/2}e^{A_6}.
}
The region $\shelltrunc{A_5^2T_2^{1/2}}{A_6T_2^{1/2}}{e^{-A_6^3}T_2^{1/2},A_6T_2^{1/2}}$ can be covered by $O(A_6^4)$ sets of the form $\shelltrunc\rho{2\rho}{z,2z}$, so by the pigeonhole principle there exist $\rho\in[A_5^2T_2^{1/2},A_6T_2^{1/2}]$ and $|z|\in[e^{-A_6^3}T_2^{1/2},A_6T_2^{1/2}]$ such that
\eqn{
\int_{-T_2/C_0}^0\int_{\shelltrunc\rho{2\rho}{z,2z}}e^{2|x|^2/T_2}(T_2^{-1}|\omega|^2+|\grad\omega|^2)dxdt\gtrsim T_2^{-1/2}e^{A_6/2}.
}
Therefore,
\eqn{
\int_{-T_2/C_0}^0\int_{\shelltrunc\rho{2\rho}{z,2z}}(T_2^{-1}|\omega|^2+|\grad\omega|^2)dxdt\gtrsim T_2^{-1/2}\exp(-O(\rho^2+z^2)/T_2).
}
By an analogous argument using \eqref{regularity}, the upper time limit in the integral can be shortened to
$-\frac{T_2^{5/2}}{\rho^2z}e^{-O(\rho^2+z^2)/T_2}$ which, by Young's inequality, is less than $-T_2e^{-O(\rho^2+z^2)/T_2}$. Thus
\eqn{
\int_{-T_2/C_0}^{-e^{-O(\rho^2+z^2)/T_2}T_2}\int_{\shelltrunc\rho{2\rho}{z,2z}}(T_2^{-1}|\omega|^2+|\grad\omega|^2)dxdt\gtrsim T_2^{-1/2}\exp(-O(\rho^2+z^2)/T_2).
}
The interval $[-T_2/C_0,-e^{-O(\rho^2+z^2)/T_2}T_2]$ can be covered by $O((\rho^2+z^2)/T_2)$ intervals of the form $[-2t_0,-t_0]$ so by the pigeonhole principle, there exists a $t_0\in[e^{-O(\rho^2+z^2)/T_2}T_2,T_2/C_0]$ such that
\eqn{
\int_{-2t_0}^{-t_0}\int_{\shelltrunc\rho{2\rho}{z,2z}}(T_2^{-1}|\omega|^2+|\grad\omega|^2)dxdt\gtrsim T_2^{-1/2}\exp(-O(\rho^2+z^2)/T_2).
}
Moreover, since $t_0\geq e^{-O(\rho^2+z^2)/T_2}T_2$, the spatial domain of integration can be covered by $e^{O(\rho^2+z^2)/T_2}T_2^{-3/2}\rho^2z$, which again is smaller than $e^{O(\rho^2+z^2)/T_2}$, balls of radius $t_0^{1/2}$. Therefore there exists an $x_*$ in the region $\shelltrunc\rho{2\rho}{z,2z}$ such that
\eq{\label{concentrate}
\int_{-2t_0}^{-t_0}\int_{B(x_*,t_0^{1/2})}(T_2^{-1}|\omega|^2+|\grad\omega|^2)dxdt\gtrsim T_2^{-1/2}\exp(-O(|x_*|^2)/T_2).
}
From here we apply Proposition \ref{uniquecont} to the function
\eqn{
(t,x)\mapsto\omega(-t,x_*+x)
}
on the interval $[0,1000t_0]$ with $\rho_{\text{carleman}}=C_0^{1/4}(t_0/T_2)^{1/2}|x_*|$ and $t_1=t_0$. Note that $r\leq|x_*|/C_0^{1/4}$ and $\rho_{\text{carleman}}\geq A_5^2T_2^{1/2}$ imply that $B(x_*,r)$ is contained in the region of regularity guaranteed by \eqref{regularity}. Therefore
\eq{\label{prime}
Z'\lesssim e^{-C_0^{1/2}|x_*|^2/500T_2}X'+t_0^{3/2}e^{O(C_0^{1/2}|x_*|^2/T_2)}Y'
}
where
\eqn{
X'=\int_{-1000t_0}^0\int_{B(x_*,C_0^{1/4}(t_0/T_2)^{1/2}|x_*|)}(t_0^{-1}|\omega|^2+|\grad\omega|^2)dxdt,
}
\eqn{
Y'=\int_{B(x_*,C_0^{1/4}(t_0/T_2)^{1/2}|x_*|)}|\omega(0,x)|^2t_0^{-3/2}e^{-|x-x_*|^2/4t_0}dx,
}
and, since $t_0^{1/2}\leq r/2$,
\eqn{
Z'=\int_{-2t_0}^{-t_0}\int_{B(x_*,t_0^{1/2})}(t_0^{-1}|\omega|^2+|\grad\omega|^2)dxdt.
}
By \eqref{concentrate}, we have
\eqn{
Z'\gtrsim T_2^{-1/2}\exp(-O(|x_*|^2)/T_2).
}
Using \eqref{regularity},
\eqn{
e^{-C_0^{1/2}|x_*|^2/500T_2}X'&\lesssim e^{-C_0^{1/2}|x_*|^2/500T_2}C_0^{3/4}t_0^{3/2}T_2^{-7/2}|x_*|^3\lesssim e^{-C_0^{1/2}|x_*|^2/1000T_2}T_2^{-1/2}.
}
Therefore within \eqref{prime}, the $X'$ term is negligible compared to the $Z'$ term and we are left with
\eqn{
\int_{B(x_*,C_0^{1/4}(t_0/T_2)^{1/2}|x_*|)}|\omega(0,x)|^2e^{-|x-x_*|^2/4t_0}dx\gtrsim \exp(-O(C_0^{1/2}|x_*|^2/T_2))T_2^{-1/2}.
}
It follows that
\eqn{
\int_{B(x_*,C_0^{-1/4}|x_*|)}|\omega(0,x)|^2dx\gtrsim \exp(-O(A_6^3))T_2^{-1/2}
}
for some $x_*$ in $\shelltrunc{A_5^2T_2^{1/2}}{2A_6T_2^{1/2}}{e^{-A_6^3}T_2^{1/2},2A_6T_2^{1/2}}$.
In conclusion,
\eq{\label{finalconc}
\int_{\shelltrunc{A_5T_2^{1/2}}{A_6^2T_2^{1/2}}{A_6^2T_2^{1/2}}}|\omega(0,x)|^2dx\gtrsim \exp(-O(A_6^3))T_2^{-1/2}
}
for all $T_2\in[A_4^2N_0^{-2},A_4^{-1}T]$.
If instead of \eqref{casewon} we had \eqref{casetoo}, then \eqref{finalconc} is immediate.

Next we convert \eqref{finalconc} back into a lower bound on the velocity. By the pigeonhole principle, there exists an $x_*$ in $\shelltrunc{A_5T_2^{1/2}}{A_6^2T_2^{1/2}}{A_6^2T_2^{1/2}}$ where
\eqn{
|\omega(0,x_*)|\gtrsim\exp(-O(A_6^3))T_2^{-1}.
}
The gradient estimate in \eqref{regularity} implies that this concentration persists up to a distance of at least $\exp(-O(A_6^3))T_2^{1/2}$ from $x_*$, and therefore
\eqn{
\left|\int_{\mathbb R^3}\omega(0,x_*-\rho y)\phi(y)dy\right|\gtrsim\exp(-O(A_6^3))T_2^{-1}
}
for a bump function $\phi$ supported in $B(0,1)$, for some $\rho=\exp(-O(A_6^3))T_2^{1/2}$. Then writing $\omega=\curl u$ and integrating by parts,
\eqn{
\left|\int_{\mathbb R^3}u(0,x_*-\rho y)\curl\phi dy\right|\geq\exp(-O(A_6^3))T_2^{-1/2}.
}
Then by H\"older's inequality,
\eqn{
\int_{B(0,1)}|u(0,x_*-\rho y)|^qdy\gtrsim\exp(-O(A_6^3))T_2^{-q/2}.
}
Within $B(x_*,\rho)$, since $\rho\leq\frac1{100}r(x_*)$, $r$ is comparable to $r(x_*)\in[A_5T_2^{1/2},A_6^2T_2^{1/2}]$. Therefore
\eq{\label{concentration}
\int_{B(x_*,\rho)}r^{q-3}|u(0,x)|^qdx\gtrsim\exp(-O(A_6^3)).
}
Since such an $x_*$ appears within every set $\shell{A_5^2T_2^{1/2}}{2A_6T_2^{1/2}}$, and $T_2$ can take any value in $[A_4^2N_0^{-2},A_4^{-1}T]$, there are at least $\log(TN_0^2)/\log A_6$ disjoint concentrations of the form \eqref{concentration}. Therefore
\eqn{
\frac{\log(TN_0^2)}{\log A_6}\exp(-O(A_6^3))\lesssim\int_{\Rt}r^{q-3}|u(0,x)|^qdx\leq A^q
}
by \eqref{concentration} and \eqref{critical}, and the desired conclusion follows.
\end{proof}

\begin{proof}[Theorem \ref{regtheorem}]
Once again, we can roughly follow \cite{tao}, but we must be a bit more careful due to our slightly worse control of $\ulin_n$. By increasing $A$, we can make $A\geq C_0$. By rescaling, it suffices to prove the theorem with $t=1$. Proposition \ref{main} implies that
\eq{\label{great}
\|P_Nu\|_{L_{t,x}^\infty([1/2,1]\times\Rt)}\leq A_1^{-1}N
}
whenever $N\geq N_*$, where
\eqn{
N_*=\exp(\exp(A_7)).
}
We apply the decomposition $u=\ulin_n+\unlin_n$ on $[0,1]$ so that on $[1/2,1]$, we have all the estimates from Proposition \ref{hierarchy}. Taking the curl, we analogously have $\omega=\wlin_n+\wnlin_n$ and define
\eqn{
E(t)=\frac12\int_{\Rt}|\wnlin_n(t,x)|^2dx
}
where we fix an $n$ sufficiently large so that \eqref{unlinq>3} gives bounds on $\unlin_n$ for $p\in[\min(q',\frac q2),3)$.

With \eqref{vorticity} and integration by parts, we compute
\eqn{
\frac d{dt}E(t)=-Y_1+Y_2+Y_3+Y_4+Y_5+Y_6+Y_7+Y_8
}
where
\eqn{
Y_1(t)&=\int_\Rt|\grad\wnlin_n|^2dx,\\
Y_2(t)&=-\int_\Rt\wnlin_n\cdot(\ulin_n\cdot\grad\wlin_n)dx,\\
Y_3(t)&=-\int_\Rt\wnlin_n\cdot(\unlin_n\cdot\grad\wlin_n)dx,\\
Y_4(t)&=\int_\Rt\wnlin_n\cdot(\wnlin_n\cdot\grad\unlin_n)dx,\\
Y_5(t)&=\int_\Rt\wnlin_n\cdot(\wnlin_n\cdot\grad\ulin_n)dx,\\
Y_6(t)&=\int_\Rt\wnlin_n\cdot(\wlin_n\cdot\grad\unlin_n)dx,\\
Y_7(t)&=\int_\Rt\wnlin_n\cdot(\wlin_n\cdot\grad\ulin_n)dx,\\
Y_8(t)&=-\int_\Rt\wnlin_n\cdot\curl(\ulin_n\cdot\grad\ulin_n)dx.
}
By H\"older's inequality, \eqref{unlinq>3}, and \eqref{ulin-bound}, we have for $t\in[1/2,1]$
\eqn{
|Y_2(t)|&\lesssim\|\wnlin_n\|_{L_t^\infty L_x^{p'}([1/2,1]\times\Rt)}\|\ulin_n\|_{L_t^\infty L_x^{2p}([1/2,1]\times\Rt)}\|\grad\wlin_n\|_{L_t^\infty L_x^{2p}([1/2,1]\times\Rt)}\\
&\lesssim A^{O(1)},
}
taking $p=\max(q,\frac q{q-2})$. The same argument applies for $Y_7$ and $Y_8$. For $Y_3$, by H\"older's inequality and \eqref{ulin-bound}, we have
\eqn{
|Y_3(t)|\lesssim E(t)^{1/2}\|\unlin_n\|_{L_t^\infty L_x^{2}([1/2,1]\times\Rt)}\|\grad\wlin_n\|_{L_{t,x}^\infty([1/2,1]\times\Rt)}\lesssim E^{1/2}A^{O(1)}.
}
For $Y_5$, H\"older's inequality and \eqref{ulin-bound} easily give
\eqn{
|Y_5(t)|\lesssim E(t)\|\grad\ulin_n\|_{L_{t,x}^\infty([1/2,1]\times\Rt)}\lesssim E(t)A^{O(1)}.
}
The same is true for $Y_6$, since Plancherel's theorem and incompressibility imply
\eqn{
\|\grad\unlin_n\|_{L_x^2(\Rt)}\lesssim\|\wnlin_n\|_{L_x^2(\Rt)}.
}
From here one proceeds to estimate $Y_4$ and conclude exactly as in \cite{tao}, making use of \eqref{integral} and \eqref{great}. (In that paper the analogous term is called $Y_3$.)
\end{proof}

As in \cite{tao}, Theorem \ref{blowuptheorem} follows immediately from Theorem \ref{regtheorem} combined with essentially any classical blowup criterion.

\bibliographystyle{abbrv}
\bibliography{references}

\end{document}